\documentclass[times,doublespace]{amsart}
\usepackage{latexsym,a4,mathrsfs,amsthm,amsmath,amssymb}
\usepackage[english]{babel}
\newtheorem{theorem}{Theorem}[section]
\newtheorem{lemma}[theorem]{Lemma}
\newtheorem{definition}[theorem]{Definition}

\newtheorem{corollary}{Corollary}[theorem]
\theoremstyle{definition}
\numberwithin{equation}{section}
\newcommand{\rw}{\rightarrow}

\newcommand{\Rr}{\mathbb{R}}
\newcommand{\C}{\mathbb{C}}
\newcommand{\Zr}{\mathbb{Z}}
\newcommand{\Nr}{\mathbb{N}}
\newcommand{\eps}{\epsilon}
\newcommand{\la}{\lambda}
\newcommand{\vth}{\vartheta}

\newcommand{\nth}[1]{\left[#1\right]_n}

\newcommand{\set}[1]{\left\{#1\right\}}
\newcommand{\E}[1]{\mathbb{E}\left[#1\right]}
\newcommand{\En}[1]{\mathbb{E}_{n}\left[#1\right]}

\newcommand{\EAn}[1]{\mathbb{E}_{\mathcal{A}_n}\left[#1\right]}
\newcommand{\Pb}[1]{\mathbb{P}\left[#1\right]}
\newcommand{\slan}{\sum_{\la \vdash n}\frac{1}{z_\la}}
\newcommand{\sla}{\sum_{\la}\frac{1}{z_\la}}
\begin{document}

\title[Randomized class functions on the symmetric groups]{On averages of randomized class functions on the symmetric groups and their asymptotics}


\author{Paul-Olivier Dehaye and Dirk Zeindler}

\address{%
Department of Mathematics, ETH Z\"urich, R\"amistrasse 101, 8092 Z\"urich, Switzerland
and
Department of Mathematics, Universit\"at Z\"urich, Winterthurerstrasse 190, 8057 Z\"urich, Switzerland
}

\begin{abstract}
The second author had previously obtained explicit generating functions for moments of characteristic polynomials of permutation matrices ($n$ points).
In this paper, we generalize many aspects of this situation.
We introduce random shifts of the eigenvalues of the permutation matrices, in two different ways: independently or not for each subset of eigenvalues associated to the same cycle.
We also consider vastly more general functions than the characteristic polynomial of a permutation matrix, by first finding an equivalent definition in terms of cycle-type of the permutation.
We consider other groups than the symmetric group, for instance the alternating group and other Weyl groups.
Finally, we compute some asymptotics results when $n$ tends to infinity.
This last result requires additional ideas: it exploits properties of the Feller coupling, which gives asymptotics for the lengths of cycles in permutations of many points.
\end{abstract}

\maketitle

\tableofcontents

\section{Introduction}
The study of spectral properties of random matrices has gained importance in many areas of mathematics and physics. In particular, the study of the spectrum or the characteristic polynomial of a random matrix in a compact Lie group has proved central in obtaining conjectures in number theory (see, for instance, the book~\cite{miller} and many papers in its reference list). 

Our initial motivation for this work was to consider another ensemble of random matrices, this time discrete: the group of permutation matrices on $n$ points. Since the spectrum of a permutation matrix is completely determined by the cycle structure of the corresponding permutation, for which many results are known, this setting allows for more precise results. Wieand has already studied \cite{WieandThesis,WieandPaper1} the fluctuation of the number of eigenvalues in an arc of the unit circle. Hambly, Keevash, O'Connell and Stark \cite{HKOS} have obtained a central limit theorem for the asymptotic value (in $n$) of the characteristic polynomial of a permutation matrix, while Zeindler \cite{DZ} obtained explicit generating functions for those values, even when evaluated inside the unit circle. The present paper simplifies some of the proofs given in the latter paper, but more importantly extends the results in many different ways: 
\begin{itemize}
\item We completely break away from the interpretation in terms of permutation matrices, and instead consider the permutations (and their cycle-type decompositions) themselves. This allows us to reexpress the characteristic polynomial associated to a permutation $\sigma$. It is now given, as a function of $x$, as the product over the cycles of $\sigma$ of $(1-x^l)$, where $l$ is the length of the cycle. This break actually happens from the start in the paper, but thinking in terms of matrices is helpful to find the natural generalizations to consider. 
\item We randomize the permutation matrices by replacing the 1s in the matrices with iid variables on the unit circle. This rotates the eigenvalues of the permutation matrix by random iid angles. We investigate both natural ways to do this: one can shift every eigenvalue independently of the other, or only shift independently each block of $l$ evenly spaced eigenvalues (each such block corresponds to a $l$-cycle present in the permutation). This is studied concurrently by Nikeghbali and Najnudel \cite{NN}. 
\item Instead of only considering the characteristic polynomial of a permutation matrix associated to the permutation $\sigma$, which is associated to a product of $(1-x^l)$ as explained above, we can replace this by a more complicated~$f(x^l)$. This defines what we call the \emph{class function associated to $f$}. We first manage to replace $f$ with a polynomial in $x$, then with a holomorphic function. We also pass onto the multivariable case $x_i$. We only present the results in two variables however, as this is necessary and sufficient to indicate the full generalization of the results and proofs. A very special case of this construction is given by the Ewens measure, which corresponds to rescalings of $f$ (see \cite{barbour}). 
\item We change the group considered, from $\mathcal{S}_n$ to $\mathcal{A}_n$ and groups closely related that appear as Weyl groups of compact Lie groups.
\end{itemize}
We give explicit generating functions for the moments over $\mathcal{S}_n$ ($\mathcal{A}_n$, etc), summing over $n$ (Theorem~\ref{thm_gen_fn_poly} for polynomial $f$ and theorem~\ref{thm_gen_holo} for holomorphic $f$).  The combinatorics always relies heavily on lemmas~\ref{lem_gen_poly} and the more general \ref{lem_gen_allg}. 

In the $\mathcal{S}_n$ case, when $|x_i|<1$, we also succeed in computing asymptotics for $n \rightarrow \infty$. The computation in the case of holomorphic $f$ (theorem~\ref{thm_conv_x<1_holo}) is actually much more difficult than for polynomial $f$ (theorem~\ref{thm_conv_x<1_poly}), and requires the introduction of the Feller coupling, a probabilistic result giving asymptotic distribution of cycle lengths of permutations in symmetric groups. 

This paper is structured as follows. In section~\ref{sec_definitions}, we introduce the combinatorial definitions needed to work with symmetric groups. In section~\ref{sec_intro_to_gen_fn}, we present the main combinatorial lemmas, which will allows us to pass from sums over conjugacy classes to products. In section~\ref{sec_wreath_prod_poly}, we start with characteristic polynomials of permutation matrices and generalize these objects to products over eigenvalues of polynomial functions and then introduce more randomness in the picture. We give some examples, and finish by computing some asymptotics for $n \rightarrow \infty$. In section~\ref{sec_extension_to_holo}, we extend definitions and results to the case of holomorphic functions. Note however that the problem of asymptotics becomes much more complicated and is thus relegated to section~\ref{sec_asym_for_holo}. That section includes the definition of the Feller coupling (in \ref{sec_feller2}) and  makes an ansatz for the asymptotic, for which we compute the moments (in \ref{sec_the_limit}).  In \ref{sec_conv_for_x<1_holo}, we prove the convergence result. We finish the paper with results about other groups than $\mathcal{S}_n$ in section~\ref{sec_other_groups}.

\section{Definitions and Notation}
\label{sec_definitions}
We introduce here notation focused on the combinatorics of the symmetric group. We leave the definitions associated to the Feller coupling for section~\ref{sec_feller2}, and to other groups for section~\ref{sec_other_groups}.
\begin{definition}
\label{def_Sn}
The \emph{permutation group $\mathcal{S}_n$} is defined to be the set of all permutations of the set $\set{1,2,\cdots,n}$.
\end{definition}
\begin{definition}
\label{def_haar_allg}
Let $G$ be a finite group. The \emph{Haar-measure} $\mu_{G}$ on $G$ is defined as
$$\mu_{G}(A):=\frac{|A|}{|G|}\text{ for every }A\subset G.$$
In the special case of $G=\mathcal{S}_n$, we write
$\En{f} := \frac{1}{n!}\sum_{\sigma  \in \mathcal{S}_n} f(\sigma)$ for the expectation on $\mathcal{S}_n$ with respect to the Haar-measure $\mu_{\mathcal{S}_n}$.
\end{definition}
We will only work with class functions on $\mathcal{S}_n$ (i.e.~$f(hgh^{-1})=f(g)$). We therefore reformulate $\En{f}$ by first parametrizing the conjugation classes of $\mathcal{S}_n$ with partitions of $n$.
\begin{definition}
 \label{def_part}
A \emph{partition} $\la$ is a  sequence of non-negative integers $\la_1 \ge \la_2 \ge \cdots$ eventually trailing to 0s, which we usually omit. The \emph{length} of $\la$ is the largest $l$ such that $\la_l \ne 0$. We define the \emph{size} $|\lambda|:= \sum_i \lambda_i$ and we set for $n\in \Nr$
$$\la\vdash n:=\set{\la\text{ partition };|\la|=n}.$$
\end{definition}
Let $\sigma\in \mathcal{S}_n$ be arbitrary. We can write $\sigma=\sigma_1\cdots \sigma_l$ with $\sigma_i$ disjoint cycles of length $\la_i$. Since disjoint cycles commute, we can assume that $\la_1\geq\la_2\geq\cdots \geq\la_l$.
We call the partition $\la=(\la_1,\la_2,\cdots,\la_l)$ the \emph{cycle-type} of $\sigma$ and write $\mathcal{C}_\la$ for the subset of $\mathcal{S}_n$ with cycle-type $\la=(\la_1,\la_2,\cdots,\la_l)$.

%
%
%
Two elements $\sigma,\tau\in \mathcal{S}_n$ are conjugate if and only if $\sigma$ and $\tau$ have the same cycle-type, and so the sets $\mathcal{C}_\la$ with $|\la|=n$ are the conjugation classes of $\mathcal{S}_n$ (see, as for much of this material on symmetric groups, \cite{macdonald} for instance).
The cardinality of each $C_\la$ is given by
$$ |\mathcal{C}_\la|=\frac{|\mathcal{S}_n|}{z_\la} \text{ with } z_\la:=\prod_{r=1}^{n} r^{c_r}c_r!$$
%
%
We put all this information together and get
\begin{lemma}
Let $f:\mathcal{S}_n\rw \C$ be a class function. Then
\begin{align}\label{E_f_class_n}
\En{f}
=
\slan f(\la).
\end{align}
\end{lemma}
%
%
We now give a construction for class functions, which we will generalize later but which already covers interesting examples.
\begin{definition}
\label{def_associated_class_fn}
Let $f(x) = \sum_{k=0} b_k x^k$ be a polynomial. Then, for a partition $\la$, set
\begin{align}
\label{eq_class_fn}
 f_\la(x)  :=  \prod_{m=1}^{l(\la)}f(x^{\la_m}).
\end{align}
This defines a class function on all $\mathcal{S}_n$ (since it only depends on the cycle-type).
We call this function the \emph{class-function associated} to $f$.
\end{definition}
We will see in section~\ref{sec_char_poly} that this includes the example of the characteristic polynomial of a permutation matrix in $\mathcal{S}_n$, when $f(x)=1-x$.
The polynomial $f$ could be taken to be a holomorphic function instead, as will be done in section~\ref{sec_extension_to_holo}. We could also introduce random variables instead of a complex number $x$, as in section~\ref{sec_Wprod_Zn}.
Definition~\ref{def_associated_class_fn} can be generalized to the case of several variables $x_1, \cdots, x_p$, with
$$
f(x_1,\cdots,x_p) = \sum_{i_1,\cdots,i_p =0} b_{i_1,\cdots,i_p} x_1^{i_1} \cdots x_p^{i_p}
$$
and
$$
f_\la(x_1,\cdots,x_p) = \prod_{m=1}^{l(\la)}  f(x_1^{\la_m},\cdots,x_p^{\la_m} ).
$$
For simplicity, we later restrict the statements and proofs to the case $p=2$.
\section{Lemmas for Generating Functions}
\label{sec_intro_to_gen_fn}
We define in this section generating functions and give two lemmas.
\begin{definition}
\label{def_gen_fn}
Let $(h_n)_{n\in\Nr}$ with $h_n\in\C$ be given.
Then the formal power series $h(t):=\sum_{n\in\Nr} h_n t^n$ is called \emph{the generating function} of the sequence $(h_n)$.
\end{definition}
\begin{definition}
\label{def_nth}
Let $h(t)=\sum_{n=0}^\infty h_n t^n$ be given. We define
$\nth{h(t)}:=h_n$.
\end{definition}
Note that we will first only need the formal aspects of these generating functions, up to the point where we use for the function  $f$ a holomorphic function (section~\ref{sec_extension_to_holo}) instead of a polynomial. After that, analytic properties in the variable $t$ will start to play a role, and we will have to be careful with radii of convergence of those power series.\\
We use here two tools for writing down generating functions.
The first is
\begin{lemma}
\label{lem_symm_fn}
Let $(a_m)_{m\in\Nr}$ be a sequence of complex numbers. Define
$$a_\la:=\prod_{m=1}^{l(\la)} a_{\la_m}.$$
Then,
\begin{align}
\label{eq_lem_symm_fn}
\sum_{\la} \frac{1}{z_\la} a_\la t^{|\la|}=\exp\left(\sum_{m=1}\frac{1}{m} a_m t^m\right).
\end{align}
Moreover, if the RHS or the LHS of \eqref{eq_lem_symm_fn} is absolutely convergent then so is the other.
\end{lemma}
\begin{proof}
The first part can be found in \cite{macdonald} or directly verified using the definitions of $z_\la$ and the exponential series. The second statement follows from applying the dominated convergence theorem at each relevant $t$.
\end{proof}

The second tool is
\begin{lemma}
\label{lem_gen_poly}
Let $f(x_1,x_2)$ be a polynomial with
$$f(x_1,x_2)=\sum_{k_1,k_2=0}^\infty b_{k_1, k_2} x_1^{k_1} x_2^{k_2}$$ and $f_\la$ its associated  class function.  We have for $|t| <1$ and $|x_i| \leq 1$

\begin{align}\label{eq_gen_poly}
\sla f_\la(x_1,x_2) t^n
=\prod_{k_1=0}^\infty \prod_{k_2=0}^\infty(1-x_1^{k_1}x_2^{k_2}t)^{-b_{k_1,k_2}},
\end{align}
and both sides of \eqref{eq_gen_poly} are holomorphic. We use the principal branch of logarithm to define $z^s$ for $z\in \C\setminus{\Rr_-}$.
\end{lemma}
%
%
\begin{proof}
We use lemma~\ref{lem_symm_fn} with $a_m=f(x_1^m,x_2^m)$ and get
\begin{align*}
\sum_{\la} \frac{1}{z_\la} \prod_{i=1}^{l(\la)}f_\la(x_1,x_2)t^n
&=
\exp\left(\sum_{m=1}\frac{1}{m} f(x_1^m,x_2^m) t^m \right)
=\exp\left(\sum_{k_1,k_2=0}^\infty b_{k_1,k_2} \sum_{m=1}\frac{t^m}{m} (x_1^{k_1}x_2^{k_2})^m\right)\\
&=\exp\left(\sum_{k_1,k_2=0}^\infty b_{k_1,k_2} (-1)\log(1-x_1^{k_1}x_2^{k_2} t)\right)
=\prod_{k_1,k_2=0}^\infty (1-x_1^{k_1}x_2^{k_2}t)^{-b_{k_1,k_2}}.
\end{align*}
Unlike later, the exchange of the sums in the second equality is immediate as only finitely many $b_{k_1,k_2}\neq 0$.
\end{proof}
\section{Randomized Class Functions associated to Polynomials}
\label{sec_wreath_prod_poly}
We want to randomize the construction of associated class function. To motivate our definitions, we base ourselves on the special case of characteristic polynomials of permutation matrices.
\subsection{The characteristic polynomial of $\mathcal{S}_n$}
\label{sec_char_poly}
We first identify $\mathcal{S}_n$ with the subgroup of permutation matrices of the unitary group $U(n)$  as follows:
\begin{align*}
\mathcal{S}_n    &\rw  U(n)\\
\sigma &\rw  (\delta_{i,\sigma(j)})_{1\leq i,j\leq n}
\end{align*}
It is easy to see that this map is an injective group homomorphism. In this section, we use the notation $g\in \mathcal{S}_n$ for matrices and $\sigma\in \mathcal{S}_n$ for permutations. The identification allows to define the characteristic polynomial of elements of $\mathcal{S}_n$ as
\begin{align}\label{eq_def_zn_det}
Z_n(x)=Z_n(x)(g):=\det(I-xg)
\end{align}
for $x\in \mathbb{C}$ and $g \in \mathcal{S}_n$. Note that this is a class function, since
$$Z_n(x)(hgh^{-1})=\det(I-xhgh^{-1})=\det(h(I-xg)h^{-1})=\det{(I-xg)}=Z_n(x)(g).$$
\begin{lemma}
Let $g \in \mathcal{S}_n$ have cycle-type $(\la_1,\cdots,\la_l)$. Then,
\begin{align}
\label{eq_def_zn_la}
Z_n(x)(g)=\prod_{m=1}^l (1-x^{\la_m}).
\end{align}

\end{lemma}
\begin{proof}
The proof follows from the simple case of $\la = (\la_1)$, i.e.~ the case of a one-cycle permutation, and observing that the characteristic polynomial factors when the permutation matrix decomposes into  blocks. More explicit details can be found in~\cite{DZ}.
\end{proof}

This shows that the characteristic polynomial of permutation matrices is the class function associated to the polynomial $1-x$, as promised in section~\ref{sec_definitions}. We will now see how to introduce some randomness in this construction, for general $f$.
\subsection{Definition of randomized class functions}
\label{sec_Wprod_Zn}
 In the special case of permutation matrices, we have at least two options: we could replace all the 1s with iid variables or only introduce one new iid variable for each cycle. We describe here the two possibilities, starting with the second option.

\subsubsection{One new variable per cycle}
\begin{definition}
\label{def_W1_Zn}
Let $\theta$ be a random variable with values in $S^1$. We set for $\la\vdash n$ and $g\in C_\la$
\begin{align}\label{eq_def_W1_Zn}
W^1Z_n(x)=W_{\theta}^1 Z_n(x)(g):=\prod_{m=1}^{l(\la)}(1-\theta_m x^{\la_m}),
\end{align}
with $\theta_m \stackrel{d}{=}\theta$, $\theta_m$ iid and independent of $g$.
\end{definition}
The fact that this corresponds to introducing one new variable per cycle can be deduced by comparison with  \eqref{eq_def_zn_la}.
We use the letter $W$ because we originally thought of this as a characteristic polynomial of the wreath product $S^1 \wr \mathcal{S}_n$ in the special case $f(x) = 1-x$.
We generalize this to arbitrary polynomials in $\C[x_1,x_2]$.
\begin{definition}
\label{def_W1_poly}
Let $\theta$ and $\vth$ be random variables with values in $S^1$ and $P$ be a polynomial with
$$P(x_1,x_2)=\sum_{k_1,k_2=0}^{\infty}b_{k_1,k_2} x_1^{k_1} x_2^{k_2}. $$
We set for $g\in C_\la$
\begin{align}\label{eq_def_W1_poly_in_two}
W^1(P)(x_1,x_2)&=W_{\theta,\vth}^{1,n} (P)(x_1,x_2)(g)
:=
\prod_{m=1}^{l(\la)}P\left(\theta_m x_1^{\la_m},\vth_m x_2^{\la_m}\right)
\end{align}
with $(\theta_m,\vth_m) \stackrel{d}{=}(\theta,\vth)$, $(\theta_m,\vth_m)$ iid and independent of $g$. This defines the \emph{first randomized class function} (of the variable $g$) \emph{associated to} the polynomial~$P$.
%
%
%
We also set
\begin{align}\label{eq_def_W1_poly}
W^1(P_1,P_2)(x_1,x_2):=W^1(P)(x_1,x_2)
\end{align}
with $P(x_1,x_2):=P_1(x_1)P_2(x_2)$.
\end{definition}
We are primarily interested in class functions of the form $W^1(P_1,P_2)(x_1,x_2)$. We have introduced this more complicated definition because we need it in section~\ref{sec_asym_for_holo}.
\subsubsection{One new variable per point}
Let $D$ be a $n\times n$ diagonal matrix with iid variables $\theta_i$  on the diagonal.
\begin{definition}
\label{def_W2_Zn_det}
We set for $g\in \mathcal{S}_n$
\begin{align}\label{eq_def_W2_Zn_det}
W^2Z_n(x)=W_\theta^2 Z_n(x)=W_\theta^2 Z_n(x)(g):=\det(I-xDg).
\end{align}
\end{definition}
An explicit computation shows that
\begin{align}\label{eq_W2_Zn_part}
W^2Z_n(x)=\prod_{m=1}^{l(\la)}\left(1-x^{\la_m} \prod_{i=1}^{\la_m}\theta_i^{(m)}\right) \text{ for }g\in C_\la
\end{align}
with $\set{\theta_i^{(m)}; 1\leq m\leq l(\la), 1\leq i \leq \la_m}=\set{\theta_1,\cdots,\theta_n}$.

This again generalizes to arbitrary polynomials.
\begin{definition}
\label{def_W2_poly}
Let $\theta$ and $\vth$ be a random variables with values in $S^1$ and $P$ be polynomials with
$$P(x_1,x_2)=\sum_{k_1,k_2=0}^{\infty}b_{k_1,k_2} x_1^{k_1} x_2^{k_2} $$
We set for $g\in C_\la$
\begin{align}\label{eq_def_W2_poly_in_two}
W^2(P)(x_1,x_2)
&=
W_{\theta,\vth}^{2,n} (P)(x_1,x_2)(g)
:=
\prod_{m=1}^{l(\la)} P\left( x_1^{\la_m}\prod_{i=1}^{\la_m}\theta_i^{(m)}, x_2^{\la_m}\prod_{i=1}^{\la_m}\vth_i^{(m)}\right)
\end{align}
with $(\theta_i^{(m)},\vth_i^{(m)}) \stackrel{d}{=} (\theta,\vth)$,
$(\theta_i^{(m)},\vth_i^{(m)})$ iid and independent of $g$. This defines the \emph{second randomized class function} (of the variable $g$) \emph{associated to} the polynomial~$P$.
We also define
\begin{align}\label{eq_def_W2_poly}
W^2(P_1,P_2)(x_1,x_2):=W^2(P)(x_1,x_2),
\end{align}
with $P(x_1,x_2):=P_1(x_1)P_2(x_2)$.
\end{definition}
\subsection{Generating functions for $W^1$ and $W^2$}
\label{sec_gen_fn_poly}
We prove in this subsection
\begin{theorem}
\label{thm_gen_fn_poly}
Let $\theta$ and $\vth$ be random variables with values in $S^1$ and $P$ be a polynomial with
$$P(x_1,x_2)=\sum_{k_1,k_2=0}^{\infty}b_{k_1,k_2} x_1^{k_1} x_2^{k_2}. $$
We define
\begin{align}\label{eq_def_alpha}
\alpha_{k_1,k_2}:=\E{\theta^{k_1}\vth^{k_2}}.
\end{align}
Then,
\begin{align}\label{eq_thm_gen_W1_poly}
\En{W^1(P)(x_1,x_2)}
&=
\nth{\prod_{k_1,k_2=0}^\infty (1-x_1^{k_1} x_2^{k_2}t)^{-b_{k_1,k_2} \alpha_{k_1,k_2}} }\\
\label{eq_thm_gen_W2_poly}
\En{W^2(P)(x_1,x_2)}
&=
\nth{\prod_{k_1,k_2=0}^\infty (1-\alpha_{k_1,k_2} x_1^{k_1} x_2^{k_2}t)^{-b_{k_1,k_2}} }.
\end{align}
These identities of coefficients can be obtained by expanding formally the (finite) products, but the products in \eqref{eq_thm_gen_W1_poly}, \eqref{eq_thm_gen_W2_poly} are actually holomorphic for $|t|<1$ and $\max\set{|x_1|,|x_2|}\leq 1$.
\end{theorem}
\textbf{Remark: }Thanks to this theorem, we can also compute the generating functions of expressions of the type $\En{\frac{\textup{d}}{\textup{d}x_1}W^1(P)(x_1,x_2)},\En{\left(\frac{\textup{d}}{\textup{d}x_1}\right)^2 W^1(P)(x_1,x_2)},\cdots$. One simply has to apply the differential operator to the products in \eqref{eq_thm_gen_W1_poly} and \eqref{eq_thm_gen_W2_poly}, after proving appropriate convergence results.
\begin{proof}[Proof of theorem~\ref{thm_conv_x<1_poly}]
The main ingredients of this proof are equation~\eqref{E_f_class_n} and lemma~\ref{lem_gen_poly}.
\begin{description}
\item[Proof of \eqref{eq_thm_gen_W1_poly}.]

We first give an expression for $\En{W^1(P)(x_1,x_2)}$ with~\eqref{E_f_class_n}:
\begin{align}
\En{W^1(P)(x_1,x_2)}\label{eq_W1_explicit_simple}
&=
\slan \E{\prod_{m=1}^{l(\la)}P\left(\theta_m x_1^{\la_m}, \vth_m x_2^{\la_m}\right)}
&=
\slan \prod_{m=1}^{l(\la)} \E{P\left(\theta_m x_1^{\la_m},\vth_m x_2^{\la_m}\right)}.
\end{align}

We therefore have to calculate $\E{P\left(\theta_m x_1^{\la_m},\vth_m x_2^{\la_m}\right)}$:
\begin{align}\label{eq_E_P(x_1,x_2)}
\E{P\left(\theta_m x_1^{\la_m}, \vth_m x_2^{\la_m}\right)}
&=
\sum_{k_1,k_2=0}^\infty b_{k_1,k_2} (x_1^{\la_m})^{k_1} (x_2^{\la_m})^{k_2}\E{\theta_m^{k_1}\vth_m^{k_2}}\nonumber\\
&=
\sum_{k_1,k_2=0}^\infty b_{k_1,k_2} \alpha_{k_1,k_2} (x_1^{\la_m})^{k_1} (x_2^{\la_m})^{k_2}.
\end{align}

We set $f(x_1,x_2):=\sum_{k_1,k_2=0}^\infty b_{k_1,k_2} \alpha_{k_1,k_2} x_1^{k_1} x_2^{k_2}$ and get
\begin{align}\label{eq_En_W^1_explicit}
\En{W^1(P)(x_1,x_2)}
=\slan f_\la(x_1,x_2).
\end{align}
We now can use lemma~\ref{lem_gen_poly} for this $f$ and get
\begin{align}\label{eq_gen_W1_pol}
\sum_{n=0}^\infty \En{W^1(P)(x_1,x_2)}t^n
&=
\sum_{n=0}^\infty \left(\slan f_\la(x_1,x_2)\right) t^n 
=
\prod_{k_1,k_2=0}^\infty (1-x_1^{k_1} x_2^{k_2}t)^{-b_{k_1,k_2} \alpha_{k_1,k_2}}.
\end{align}
We have therefore found a generating function for $W^1(P)(x_1,x_2)$.
\item[Proof of \eqref{eq_thm_gen_W2_poly}.]
The calculations are very similar. The only difference is that:
\begin{multline*}
\E{P\left(x_1^{\la_m}\prod_{i=1}^{\la_m}\theta_i^{(m)}, x_2^{\la_m}\prod_{i=1}^{\la_m}\vth_i^{(m)}\right)}
\\=
\sum_{k_1,k_2=0}^\infty b_{k_1,k_2}(x_1^{\la_m})^{k_1} (x_2^{\la_m})^{k_2}\prod_{i=1}^{\la_m}\E{(\theta_i^{(m)})^{k_1}(\vth_i^{(m)})^{k_2}}\\
\\=
\sum_{k_1,k_2=0}^\infty b_{k_1,k_2}(x_1^{\la_m})^{k_1} (x_2^{\la_m})^{k_2}\alpha_{k_1,k_2}^{\la_m},
\end{multline*}
with $\alpha_{k_1,k_2}$ as above. Since the exponent of $\alpha_{k_1,k_2}^{\la_m}$ is dependent on $\la_m$, we have to use the several (i.e.~more-than-2) variables case of lemma~\ref{lem_gen_poly}.
Explicitly, we use
$$f(x_1,x_2,\alpha_{1,1},\cdots,\alpha_{d_1,d_2})
= \sum_{k_1,k_2=0}^\infty b_{k_1,k_2} x_1^{k_1} x_2^{k_2}\alpha_{k_1,k_2},$$
where $d_1,d_2$ is the degree of $P$ in $x_1,x_2$.
The only monomials in $f$ with non-zero coefficients have the form $x_1^{k_1} x_2^{k_2}\alpha_{k_1,k_2}$.
Therefore,
\begin{align}\label{eq_gen_W2_pol}
\sum_{n=0}^\infty \En{W^2(P)(x_1,x_2)}t^n
=
\prod_{k_1,k_2=0}^\infty (1-x_1^{k_1} x_2^{k_2}\alpha_{k_1,k_2}t)^{-b_{k_1,k_2} }.
\end{align}
We have thus found a generating function for $W^2(P)(x_1,x_2)$.
\end{description}
\end{proof}
\subsection{Examples}
\label{sec_examples_poly}
We now give some examples of generating functions that can be obtained through these results.
\subsubsection{The characteristic polynomial and $\theta\equiv\vth \equiv1$}
We now write down a generating function for $\En{Z_n^{s_1}(x_1)Z_n^{s_2}(x_2)}$.
We set $P_1(x_1)=(1-x)^{s_1}$ and $P_2(x_2)=(1-x)^{s_2}$.
Clearly $\alpha_{k_1,k_2}=1$, so $W^1=W^2$ (this of course needs not be true in general). We get
\begin{align}
\sum_{n=0}^\infty \En{Z_n^{s_1}(x_1)Z_n^{s_2}(x_2)} t^n
&=
\sum_{n=0}^\infty \En{W^1(P_1,P_2)(x_1,x_2)} t^n
=
\sum_{n=0}^\infty \En{W^2(P_1,P_2)(x_1,x_2)} t^n\nonumber\\
&=
\prod_{k_1=0}^{s_1}\prod_{k_2=0}^{s_2}\left(1-x_1^{k_1}x_2^{k_2}t\right)^{\binom{s_1}{k_1}\binom{s_2}{k_2}(-1)^{k_1+k_2+1}}
\label{eq_gen_Zn_theta=1}
\end{align}
\begin{corollary}
[already proved in \cite{DZ}]
We have for $n \geq 1$ that
\begin{align}\label{E_Z_n}
\E{Z_n(x)}=1-x.
\end{align}
\end{corollary}
%
%
%
%
%
%
%
%
%
%
\subsubsection{The case $\theta=\overline{\vth}$ uniform on $S^1$}
We have $\alpha_{k_1,k_2}=\left\{
                            \begin{array}{ll}
                              1, & k_1=k_2 \\
                              0, & \hbox{otherwise.}
                            \end{array}
                          \right.$, which again implies $W^1 = W^2$.
We get for
$$P_1(x_1)=\sum_{k=0}^{d_1} a_k x^k, \qquad P_2(x_2)=\sum_{k=0}^{d_2} b_k x^k$$ that
\begin{align}
\sum_{n=0}^\infty \En{W^1(P_1,P_2)(x_1,x_2)} t^n
&=\sum_{n=0}^\infty \En{W^2(P_1,P_2)(x_1,x_2)} t^n\nonumber\\
&= \prod_{k=0}^\infty (1-x_1 x_2 t)^{-a_k b_k}
\end{align}
and
\begin{align}
\sum_{n=0}^\infty \En{\bigl(W^1Z_n(x)\bigr)^{s_1} \overline{\bigl(W^1Z_n(x)\bigr)^{s_2}} } t^n
&=\sum_{n=0}^\infty \En{\bigl(W^2Z_n(x)\bigr)^{s_1} \overline{\bigl(W^2Z_n(x)\bigr)^{s_2}} } t^n\nonumber\\
&= \prod_{k=0}^\infty (1-|x|^{2k} t)^{-\binom{s_1}{k} \binom{s_2}{k}}.
\label{eq_example_chara_theta_uniform}
\end{align}
Equation \eqref{eq_example_chara_theta_uniform} is also valid for $|x|=1$ (see theorem~\ref{thm_gen_fn_poly}). We get in this case

\begin{multline}\label{eq_example_chara_theta_uniform_x=1}
\sum_{n=0}^\infty \En{\bigl(W^1Z_n(x)\bigr)^{s_1} \overline{\bigl(W^1Z_n(x)\bigr)^{s_2}} } t^n
\\=
\sum_{n=0}^\infty \En{\bigl(W^2Z_n(x)\bigr)^{s_1} \overline{\bigl(W^2Z_n(x)\bigr)^{s_2}} } t^n
=
(1-t)^{-\binom{s_1+s_2}{s_1}},
\end{multline}
where we have used the Vandermonde identity for binomial coefficients.
%
%
\subsubsection{An example with $\theta=\overline{\vth}$ discrete on $S^1$}
\label{ex_discrete}
We choose $\theta$ with $\Pb{\theta=e^{m\frac{2\pi i}{p}}}=\frac{1}{p}$ for $p\in\Nr$ and $\vth=\overline{\theta}$.
Then $\alpha_{k_1,k_2}=\left\{
                         \begin{array}{ll}
                           1, &p\hbox{ divides }(k_1-k_2)\\
                           0, & \hbox{otherwise}
                         \end{array}
                       \right.$, and still $W^1=W^2$:
\begin{multline}
\sum_{n=0}^\infty \En{W^1(P_1,P_2)(x_1,x_2)} t^n
=
\sum_{n=0}^\infty \En{W^2(P_1,P_2)(x_1,x_2)} t^n
\\=
\prod_{\substack{k_1,k_2=0\\p|(k_1-k_2)}}^\infty (1-x_1 x_2 t)^{-a_{k_1} b_{k_2}}.
\end{multline}
This situation is similar to what is described in \cite{WieandPaper2}.
\subsection{Asymptotics for $|x|<1$}
\label{sec_conv_for_x<1_poly}
We have found in section \ref{sec_gen_fn_poly} generating functions for both types of  class functions. We can now extract the behaviour of $\En{W^j(P)}$ for $n\rw\infty$ and $\max\set{|x_1|,|x_2|}<1$.
\begin{theorem}
\label{thm_conv_x<1_poly}
Let $x_1,x_2$ be complex numbers with $|x_i|<1$ and $P$ be a polynomial with
$$P(x_1,x_2)=\sum_{k_1,k_2=0}^{\infty} b_{k_1,k_2}x_1^{k_1}x_2^{k_2}.$$
If $b_{0,0}\notin\Zr_{\leq 0}$ then
\begin{align}
\label{eq_thm_conv_poly_w1}
\En{W^{1}(P)(x_1,x_2)} \sim
\frac{n^{b_{0,0}-1}}{\Gamma(b_{0,0})} \prod_{\substack{k_1,k_2\in\Nr\\k_1+k_2\neq0}} (1-x_1^{k_1}x_2^{k_2})^{-b_{k_1,k_2}\alpha_{k_1,k_2}}
\end{align}
and
\begin{align}\label{eq_thm_conv_poly_w2}
\En{W^{2}(P)(x_1,x_2)} \sim
\frac{n^{b_{0,0}-1}}{\Gamma(b_{0,0})} \prod_{\substack{k_1,k_2\in\Nr\\k_1+k_2\neq0}}^\infty (1-\alpha_{k_1,k_2}x_1^{k_1}x_2^{k_2})^{-b_{k_1,k_2}}.
\end{align}
%
If $b_{0,0}\in\Zr_{\leq 0}$ then we just have
\begin{align}
\En{W^1(P)(x_1,x_2)} &\rw 0 \label{eq_thm_conv_poly_w1_special},\\
\En{W^2(P)(x_1,x_2)} &\rw 0 \label{eq_thm_conv_poly_w2_special}.
\end{align}
\end{theorem}
\begin{proof}
One can prove this theorem by induction over the number of factors, but this is rather technical. A more sophisticated way is to use Cauchy's integral formula. The details of this proof can be found in \cite{FlSe09}, theorem VI.1 and VI.3.
\end{proof}

\section{Randomized Class Functions associated to Holomorphic Functions}
\label{sec_extension_to_holo}
Our goal is now to extend what we did in section~\ref{sec_wreath_prod_poly} for polynomials onto holomorphic functions. The proofs of this section thus apply to section~\ref{sec_wreath_prod_poly} as well, but they are more challenging technically: the products that were finite now become infinite, which require us to leap beyond formal generating series in $t$ and actually consider convergence issues in $t$.

The main issue arises with the extension of theorem~\ref{thm_conv_x<1_poly}: the theorem is also true for holomorphic functions, but we cannot argue anymore by induction over the number of factors, since there are now infinitely many. We could still devise a proof based on complex analysis, as in \cite{FlSe09}. We will give a different proof with probability theory. Since this needs a lot of work, we defer the extension of theorem~\ref{thm_conv_x<1_poly}  to section~\ref{sec_asym_for_holo}.

Let $x_0\in\C$ and $r\in\Rr_+$, and set $B_r(x_0):=\set{ x\in \C ; |x-x_0|<r}.$ We now extend lemma~\ref{lem_gen_poly} (again stated for the case of $p=2$ variables only):
\begin{lemma}
\label{lem_gen_allg}
Let $f(x_1,x_2)$ be a holomorphic function in $B_{r_1}(0)\times B_{r_2}(0)$  with
$$f(x_1,x_2)=\sum_{k_1,k_2=0}^\infty b_{k_1, k_2} x_1^{k_1} x_2^{k_2}$$
and $f_\lambda$ its associated class function. Set
$$\Omega':=\set{(x_1,x_2)\in \C^2; |x_i| \leq 1 \text{ if } r_i>1 \text{ and } |x_i| <   r_i \text{ if }r_i \leq1}.$$
We then have on $\Omega'\times B_1(0)$
\begin{align}\label{eq_gen_allg}
\sla f_\la(x_1,x_2) t^n
=\prod_{k_1=0}^\infty \prod_{k_2=0}^\infty(1-x_1^{k_1}x_2^{k_2}t)^{-b_{k_1,k_2}}.
\end{align}

The product is holomorphic in $(x_1,x_2,t)$ in the interior of $\Omega'\times B_1(0)$.
The product is holomorphic in $t$ in $B_1(0)$ for all $(x_1,x_2)\in\Omega'$.
Note that we use the principal branch of logarithm to define $z^s$ for $z\in \C\setminus{\Rr_-}$.
\end{lemma}
%
\begin{proof}
The proof works as the proof of lemma~\ref{lem_gen_poly}, except that the justification for the exchange of sums that occurs in the second equality is barely more tricky:
\begin{multline*}
\left|\sum_{k_1,k_2=0}^\infty \sum_{m=1}^\infty\frac{t^m}{m} b_{k_1,k_2} (x_1^{k_1}x_2^{k_2})^m\right|
\leq \sum_{k_1,k_2=0}^\infty \sum_{m=1}^\infty \left|\frac{t^m}{m}\right||b_{k_1,k_2} x_1^{k_1}x_2^{k_2}|
\\=\log(1-|t|)\sum_{k_1,k_2=0}^\infty |b_{k_1,k_2} x_1^{k_1}x_2^{k_2}|<\infty,
\end{multline*}
where the last inequality is true since $f$ is holomorphic in $B_{r_1}(0)\times B_{r_2}(0)$.
\end{proof}
\textbf{Remark}: The conditions on $x_1,x_2$ and $t$ ensures that the Taylor-expansion of $\log(1-x_1^{k_1}x_2^{k_2}t)$ is absolutely convergent.
If $f$ is a (Laurent-) polynomial, one can replace this condition by any other that ensures $\sup\set{|x_1^{k_1}x_2^{k_2}t|, b_{k_1,k_2}\neq 0}<1$.

Until now, we have defined randomized class functions of two types associated to polynomials.
We can use formula \eqref{eq_def_W1_poly_in_two} and \eqref{eq_def_W2_poly_in_two} to define $W^{j,n}(f)(x_1,x_2)$ for a holomorphic function $f$.
We will always assume in what follows that $f$ is holomorphic in $B_{r_1}(0)\times B_{r_2}(0)$.
The function $W^{j,n}(f)(x_1,x_2)$ is also holomorphic in $B_{r_1}(0)\times B_{r_2}(0)$ since $\theta,\vth\in S^1$.

We now get a complete analog of the result in section \ref{sec_gen_fn_poly}.
\begin{theorem}
\label{thm_gen_holo}
Let $f$ be a holomorphic function in $B_{r_1}(0)\times B_{r_2}(0)$ with
$$f(x_1,x_2)=\sum_{k_1,k_2=0}^\infty b_{k_1,k_2}x_1^{k_1}x_2^{k_2}$$
We define as in \eqref{eq_def_alpha} $\alpha_{k_1,k_2}:=\E{\theta^{k_1}\vth^{k_2}}$ and
$$\Omega':=\set{(x_1,x_2)\in \C^2; |x_i| \leq 1 \text{ if } r_i>1 \text{ and }\\
                                |x_i| <   r_i \text{ if }r_i \leq1}.
$$
We get
\begin{align}
\label{eq_lem_gen_holo_1}
\En{W^1(f)(x_1,x_2)}
&=
\nth{\prod_{k_1,k_2=0}^\infty (1-x_1^{k_1}x_2^{k_2}t)^{-\alpha_{k_1,k_2}b_{k_1,k_2}}}\\
\label{eq_lem_gen_holo_2}
\En{W^2(f)(x_1,x_2)}
&=
\nth{\prod_{k_1,k_2=0}^\infty (1-\alpha_{k_1,k_2}x_1^{k_1}x_2^{k_2}t)^{-b_{k_1,k_2}}}.
\end{align}
The products are holomorphic in $(x_1,x_2,t)$ in the interior of $\Omega'\times B_1(0)$.
The products are holomorphic in $t$ in $B_1(0)$ for all $(x_1,x_2)\in\Omega'$.
\end{theorem}
\begin{proof}
The proof of \eqref{eq_lem_gen_holo_1} is almost similar to the proof of \eqref{eq_thm_gen_W1_poly}.
There is only one important difference. The function $f$ defined in the proof of \eqref{eq_thm_gen_W1_poly} is now a holomorphic function and not a polynomial. We thus have to apply lemma~\ref{lem_gen_allg} instead of lemma~\ref{lem_gen_poly}. The function $f$ is holomorphic in $B_{r_1}(0)\times B_{r_2}(0)$ since $|\alpha_{k_1,k_2}|\leq 1$. This is thus sufficient to  prove \eqref{eq_lem_gen_holo_1}.

The proof of \eqref{eq_lem_gen_holo_2} is little bit more intricate. We cannot keep the analogy with the proof of \eqref{eq_thm_gen_W2_poly} and simply use lemma~\ref{lem_gen_allg}. Indeed, we would now have a holomorphic function in infinitely many variables. Thankfully, we can still use our main lemma, lemma~\ref{lem_symm_fn}. We have found in the proof of lemma~\ref{lem_gen_allg} that
\begin{align*}
\E{P\left(x_1^{\la_m}\prod_{i=1}^{\la_m}\theta_i^{(m)},x_2^{\la_m}\prod_{i=1}^{\la_m}\vth_i^{(m)}\right)}
&=
\sum_{k_1,k_2=0}^\infty b_{k_1,k_2}(x_1^{\la_m})^{k_1} (x_2^{\la_m})^{k_2}\alpha_{k_1,k_2}^{\la_m}.
\end{align*}
We define here $$a_m:=\sum_{k_1,k_2=0}^\infty b_{k_1,k_2}(x_1^{m})^{k_1} (x_2^{m})^{k_2}\alpha_{k_1,k_2}^{m}$$
and get with lemma~\ref{lem_symm_fn}
\begin{align*}
\sum_{n=0}^\infty \left( \En{W^2(f)(x_1,x_2)}\right)t^n
&=
\sum_{n=0}^\infty \left( \slan a_\la\right)t^n
=
\exp \left(\sum_{m=1}^\infty a_m t^m\right).
\end{align*}
The last steps are now completely similar to the proof of lemma~\ref{lem_gen_allg}.
\end{proof}
We now consider the generalization of theorem~\ref{thm_conv_x<1_poly}  to the case of holomorphic functions.
\section{Asymptotics for randomized Class Functions associated to Holomorphic Functions}
\label{sec_asym_for_holo}
We assume as in the last section that $f$ is holomorphic in $B_{r_1}(0)\times B_{r_2}(0)$.
We have calculated in theorem~\ref{thm_conv_x<1_poly} the behaviour of $\En{W^j(P)}$ for $n\rw\infty$. It is natural to ask if this lemma generalizes to class functions associated to holomorphic functions. Explicitly, we prove
\begin{theorem}
\label{thm_conv_x<1_holo} Let $x_1,x_2\in\C$ be given with $|x_i| < \min(r_i,1)$ and $f,\alpha_{k_1,k_2}$ be as in theorem \ref{thm_gen_holo}.

If $b_{0,0}\notin\Zr_{\leq 0}$, then
\begin{align}\label{eq_lem_conv_holo_w1}
\En{W^{1}(f)(x_1,x_2)} \sim
\frac{n^{b_{0,0}-1}}{\Gamma(b_{0,0})} \prod_{\substack{k_1,k_2\in\Nr\\k_1+k_2\neq0}} (1-x_1^{k_1}x_2^{k_2})^{-b_{k_1,k_2}\alpha_{k_1,k_2}}
\end{align}
and
\begin{align}\label{eq_lem_conv_holo_w2}
\En{W^{2}(f)(x_1,x_2)} \sim
\frac{n^{b_{0,0}-1}}{\Gamma(b_{0,0})} \prod_{\substack{k_1,k_2\in\Nr\\k_1+k_2\neq0}} (1-\alpha_{k_1,k_2}x_1^{k_1}x_2^{k_2})^{-b_{k_1,k_2}}.
\end{align}
If $b_{0,0}\in\Zr_{\leq 0}$ then there exists a $\delta=\delta(x_1,x_2)>0$ such that
\begin{align}
\En{W^1(f)(x_1,x_2)} & = O\left(\frac{1}{(1+\delta)^n}\right) \qquad (n\rw\infty)\label{eq_lem_conv_holo_w1_special},\\
\En{W^2(f)(x_1,x_2)} & = O\left(\frac{1}{(1+\delta)^n}\right) \qquad (n\rw\infty) \label{eq_lem_conv_holo_w2_special}.
\end{align}
\end{theorem}
This theorem is a direct generalization of theorem~\ref{thm_conv_x<1_poly}.
We cannot argue anymore by induction over the number of factors, since there are infinitely many factors in the RHS of \eqref{eq_thm_conv_poly_w1} and \eqref{eq_thm_conv_poly_w2}. One can still use the proof in \cite{FlSe09}, but we give here a different proof with probability theory. The main argumentation will be based on the Feller coupling. This proof will occupy us for this whole section.
%
%

The proof of theorem~\ref{thm_conv_x<1_holo} will run as follows.  We prove in section~\ref{sec_reduction_to_1} that the case $b_{0,0}=1$ implies the general case.  We first prove theorem~\ref{thm_conv_x<1_holo} for $W^{1}(f)$ for $b_{0,0}=1$ and give at the end some comments for the proof for $W^2$.  In order to prove this, we give in section \ref{sec_easy_facts} some definitions, conventions and some easy facts. In section~\ref{sec_feller2} we give an alternative expression for $W^1(f)$ using cycles and define the Feller coupling. This allows us to compare $W^{1,n}(f)$ and $W^{1,n+1}(f)$.  After these preparations, we suggest in section~\ref{sec_the_limit} a candidate $W^{1,\infty}(f)$ for the limit in $n$ of $W^{1,n}(f)$ and prove some analytic results on it. Finally, we prove in section~\ref{sec_conv_for_x<1_holo} the convergence $\E{W^{1,n}(f)}\rw \E{W^{1,\infty}(f)} $.
%
%
%
%
\subsection{Reduction to $b_{0,0}=1$}
\label{sec_reduction_to_1}
\begin{lemma}
\label{lem_reduction_to_1}
If theorem~\ref{thm_conv_x<1_holo} is true for $b_{0,0}=1$ then it is true for all $b_{0,0}\in \mathbb{C}$.
\end{lemma}
Before we prove this lemma, we do some (small) preparations
\begin{definition}
\label{def_binom}
We set for $s\in\C, k\in \Nr$.
\begin{align}\label{eq_def_binom}
\binom{s}{k}:=\prod_{m=1}^k \frac{s-m+1}{m}=\frac{\Gamma(s+1)}{\Gamma(k+1)\Gamma(s-k+1)}\text{  and  }\binom{s}{0}:=1.
\end{align}
\end{definition}
The proof of the last equality can be found in \cite{busam}.  We then have
\begin{lemma}
\label{lem_newton_series}
We have for each $s,z\in\C$ with $|z|<1$ or $\Re(s)<0,|z|=1$ that
\begin{align}\label{eq_newton_series }
\frac{1}{(1+z)^{s}}=\sum_{k=0}^\infty \binom{s-1+k}{k} z^k,
\end{align}
and the sum is absolutely convergent in both cases. Also, we have for $s\notin \set{-1,-2,-3,\cdots}$
\begin{align}\label{eq_growth_binom}
\binom{n+s}{n} = \frac{n^{s}}{\Gamma(s+1)}\big(1+O(n^{-1})\bigr)     \qquad (n\rw\infty).
\end{align}
\end{lemma}
We also need
\begin{lemma}[Euler-MacLaurin Formula, see \cite{Apostol:99}]\label{euler_mac}
Let $a:[0,\infty]\rw\C$ be a smooth function. We then have for all $n\geq 2$
\begin{equation}\label{eq_euler_mac}
\sum_{k=2}^n a(k)
=
\int_1^n a(s) \ \text{d}s + \int_1^n (s-\lfloor s\rfloor) a^\prime(s) \ \text{d}s.
\end{equation}
 \end{lemma}
\begin{proof}[Proof of lemma~\ref{lem_reduction_to_1}]
We prove this lemma only for $W^1$, since the proof for $W^2$ is the same.  We put $c=b_{0,0}-1$ and rewrite the generating function in \eqref{eq_lem_gen_holo_1} as follows:
\begin{align}\label{eq_gen_reformulated}
\left(\frac{1}{(1-t)^{c}}\right)
\underbrace{\left(\frac{1}{1-t}\prod_{\substack{k_1,k_2=0\\ k_1+k_2 \neq 0}}^\infty(1-x_1^{k_1} x_2^{k_2} t)^{-b_{k_1,k_2} \alpha_{k_1,k_2}}\right)}_{=:h(x_1,x_2,t)}.
\end{align}
We define $\widetilde{f}(x_1,x_2):=f(x_1,x_2)-c$. Then $\widetilde{f}(0,0)=1$ and $h(x_1,x_2,t)$ is the generating function for $W^1(\widetilde{f})(x_1,x_2)$ (compare with \eqref{eq_lem_gen_holo_1}). Since we assume that theorem~\ref{thm_conv_x<1_holo} is true for $b_{0,0}=1$ and in that case the RHS of \eqref{eq_lem_conv_holo_w1} does not depend on $n$, we can write $h(x_1,x_2,t)=\sum_{n=0}^\infty h_n t^n$ with $h_n\rw h_\infty \in \mathbb{C}$.
We first look at the convergence rate of the sequence $(h_k)_{k\in\Nr}$.
Let $x_1,x_2$ be fixed. The function $(1-t)h(x_1,x_2,t)$ is holomorphic in $B_{1+\delta}(0)$ for some $\delta>0$ small enough. This can be seen by inspecting the proof of lemma~\ref{lem_gen_allg}.
The convergence radius of the expansion of $(1-t)h(x_1,x_2,t)$ around $0$ is therefore at least $1+\delta$
and so $|\nth{(1-t)h(x_1,x_2,t)}|= O\bigl( (1+\delta)^{-n}\bigr)$. But $\nth{(1-t)h(x_1,x_2,t)}= h_n-h_{n-1}$. Therefore
$
|h_n-h_{n-1}|
=
O\bigl( (1+\delta)^{-n}\bigr) 
$. We get
\begin{align}
|h_{\infty}-h_n|
&=
\lim_{m\rw\infty} |h_m-h_{n}|
\leq
\lim_{m\rw\infty}\sum_{k=n+1}^m |h_k-h_{k-1}|
=
O\left(
\sum_{k=n+1}^m (1+\delta)^{-k}
\right)\nonumber\\
&\leq 
O\left(
\sum_{k=n+1}^\infty (1+\delta)^{-k}
\right)
=
O\bigl( (1+\delta)^{-n}\bigr) 
\end{align}
We are ready to prove theorem~\ref{lem_reduction_to_1}.
\begin{description}
\item[Case 0, $c=0$ ]
This case is trivial, since $c=0$ implies $b_{0,0}=1$.
\item[Case 1, $c\in\set{-1,-2,-3,\cdots}$ ]
In this case $-c\in\Nr$ and $\binom{-c}{k} = 0$ for $k\geq -c$. We get for $n\geq -c$
\begin{align*}
\nth{\frac{1}{(1-t)^{c}}h(x_1,x_2,t)}
&=
\nth{(1-t)^{-c}h(x_1,x_2,t)}
=
\sum_{k=0}^{-c} h_{n-k} \binom{-c}{k}(-1)^k\\
&=
\sum_{k=0}^{-c}\Bigl(h_\infty + (1+\delta)^{-(n-k)} \Bigr)\binom{-c}{k}(-1)^k\\
&=
h_\infty \left(\sum_{k=0}^{-c}\binom{-c}{k}(-1)^k\right) + O\bigr((1+\delta)^{-n}\bigl)
= 
O\bigr((1+\delta)^{-n}\bigl)
\end{align*}
%
%
%
\item[Case 2, $c\notin\Zr_{\leq 0}$ ]
This case is a little bit more difficult. We have
\begin{align*}
\nth{\frac{1}{(1-t)^{c}}h(x_1,x_2,t)}
&=
\nth{(1-t)^{-c}h(x_1,x_2,t)}
=
\sum_{k=0}^{n} h_{n-k} \binom{c-1+k}{k} \\
&=
h_n + ch_{n-1} +
\sum_{k=2}^{n}\Bigl( h_{\infty}+ O\bigl( (1+\delta)^{-(n-k)}\bigr)\Bigr)\Bigl(\frac{k^{c-1}}{\Gamma(c)}+O\bigl( k^{c-2}\bigr)\Bigr)
\end{align*}
Obviously we have $h_n + ch_{n-1}= (1+c) h_\infty + O\bigl( n^{c-2}\bigr)$. \\
It follows immediately form lemma~\ref{euler_mac} (with a small calculation) that
\begin{align}
\sum_{k=2}^n \frac{k^{c-1}}{\Gamma(c)}
=
const. + \frac{1}{\Gamma(c+1)}n^c+ O\bigl( n^{c-2}\bigr)
\end{align}
Therefore the leading term gives precisely what we want. There remains to show that the other terms behaves well.
We get again with lemma~\ref{euler_mac} and $d:=\Re(c)$
\begin{align*}
&\sum_{k=2}^n \left|(1+\delta)^{-(n-k)}k^{c-1}\right|
=
\frac{1}{(1+\delta)^n} \sum_{k=2}^n k^{d-1}(1+\delta)^k\\
&=
\frac{1}{(1+\delta)^n} \left(
\int_1^n s^{d-1}(1+\delta)^s \ \text{d}s 
+ \int_1^n (s-\lfloor s\rfloor) \left(s^{d-1}(1+\delta)^s\right)^\prime \ \text{d}s
 \right).
\end{align*}
But
\begin{align*}
\left|\int_1^n s^{d-1}(1+\delta)^s \ \text{d}s\right|
&=\left|\left.\Bigl(s^{d-1}\frac{(1+\delta)^s}{\log(1+\delta)}
\Bigr)\right|_{s=1}^n
-
\int_1^n s^{d-2}\frac{(1+\delta)^s}{\log(1+\delta)} \ \text{d}s\right| \\
&\leq
\left.\Bigl(s^{d-1}\frac{(1+\delta)^s}{\log(1+\delta)}
\Bigr)\right|_{s=1}^n
+
\int_1^n s^{d-2}\frac{(1+\delta)^s}{\log(1+\delta)} \ \text{d}s \\
&\leq
\left(n^{d-1}\frac{(1+\delta)^n}{\log(1+\delta)} - \frac{(1+\delta)}{\log(1+\delta)} \right)
+
\int_1^n s^{d-2}\frac{(1+\delta)^n}{\log(1+\delta)} \ \text{d}s \\
& = 
const. + O(n^{d-1}(1+\delta)^n) 
=
 const. + O(n^{c-1} (1+\delta)^n)
\end{align*}
Therefore
$$
\sum_{k=2}^n \left|(1+\delta)^{-(n-k)}k^{c-1}\right|
=
const. +O(n^{c-1})
$$
We put everything together and get
\begin{equation}
\nth{\frac{1}{(1-t)^{c}}h(x_1,x_2,t)}
=
const. + \frac{h_\infty}{\Gamma(c+1)}n^c  +O(n^{c-1})\label{eq_asymptotic_nth_holo}
\end{equation}
If $\Re(c)> 0$, then \eqref{eq_asymptotic_nth_holo} is enough to prove lemma~\ref{lem_reduction_to_1}.
If $\Re(c)< 0$ then we have to prove that the constant is $0$. We know that the sequence $(h_k)_{k\in\Nr}$ is bounded by a constant $C$. Therefore
\begin{multline}
\left|\nth{\frac{1}{(1-t)^{c}}h(x_1,x_2,t)}\right|
=
\left|\sum_{k=0}^{n} h_{n-k} \binom{c-1+k}{k} \right|
\\\leq
C \sum_{k=0}^{n} \left| \binom{c-1+k}{k}   \right|
\leq
C \sum_{k=0}^{\infty} \left| \binom{c-1+k}{k}   \right|
< \infty
\end{multline}
by lemma~\ref{lem_newton_series}. We therefore can apply dominated convergence (with $h_k=0$ for $k<0$) and get
\begin{multline}
\lim_{n\rw\infty}\sum_{k=0}^{\infty} h_{n-k} \binom{c-1+k}{k}
=
\sum_{k=0}^{\infty} \lim_{n\rw\infty} h_{n-k} \binom{c-1+k}{k}
\\=
\sum_{k=0}^{\infty} h_\infty \binom{c-1+k}{k}
=
0
\end{multline}
The case $\Re(c)=0,c\neq 0$ is the only one left. It is easy to see that the constant in \eqref{eq_asymptotic_nth_holo} is continuous in $c$. This completes the proof.

\end{description}
\end{proof}
\subsection{Definitions, conventions, simplifications and some facts}
\label{sec_easy_facts}
We calculate as in \eqref{eq_E_P(x_1,x_2)}:
\begin{align*}
\E{f\left(\theta_m x_1, \vth_m x_2\right)}
&=
\sum_{k_1,k_2=0}b_{k_1,k_2} \alpha_{k_1,k_2} x_1^{k_1} x_2^{k_2}
=:
\widetilde{f}(x_1,x_2)
\end{align*}
 and therefore
\begin{align*}
\En{W_{\theta,\vth}^{1,n}(f)(x_1,x_2)}
=
\En{W_{1,1}^{1,n}(\widetilde{f})(x_1,x_2)}.
\end{align*}
Since we are only interested in expectations, we can assume $\theta \equiv \vth \equiv 1$ and consider $\widetilde{f}$ instead of $f$.

%
We will also assume that $f$ only depends on one variable. The arguments in the proof are the same, but the expressions are much more simple.

We now set a $0<r<\min\set{1,r_1}$ fixed and prove theorem~\ref{thm_conv_x<1_holo} for $|x|<r$.
We shrink $x$ because we sometimes need $\sup |f(x)|$ to be finite.

\subsection{Cycle notation and the Feller coupling}
\label{sec_feller2}
In one variable, for $f$ holomorphic, one of the definitions given above simplifies to
\begin{align}\label{eq_W1_theta=vth=1}
\Bigl(W^1(f)(x)\Bigr)(\sigma)
&=\prod_{m=1}^{l(\la)} f\left(x^{\la_m}\right) \text{ for }\sigma \in \mathcal{C}_\la.
\end{align}
%

%
%
%
%

Some of the factors in \eqref{eq_W1_theta=vth=1} are equal and we therefore can collect them. We do this as follows
\begin{definition}
\label{def_cm_cycle}
Let $\la=(\la_1,\cdots,\la_l)$ be a partition of $n$. We define
\begin{align}
C_m:=C_m^{(n)}:=C_m^{(n)}(\la):=\# \set{i | 1\leq i\leq l  \text{ and } \la_i = m}.
\end{align}
\end{definition}
We get with definition~\ref{def_cm_cycle} and \eqref{eq_W1_theta=vth=1}
\begin{align}\label{eq_def W1_holo_with_cycle}
\Bigl(W^1(f)(x)\Bigr)(\sigma)
&=
\prod_{m=1}^{n} \bigl(f(x^m)\bigr)^{C_m^{(n)}(\la)}\text{ for } \sigma\in \mathcal{C}_\la.
\end{align}
We interpret the functions $C_m^{(n)}$ as class functions $C_m^{(n)}: \mathcal{S}_n \rw \Nr$ and therefore as random variables on $\mathcal{S}_n$. One can find two useful lemmas in \cite{barbour}.
%
%
\begin{lemma}
\label{lem_cycle_distribution}
Let $c_1,c_2,\cdots,c_n\in \Nr$ be given with $\sum_{m=1}^n m c_m=n$. Then
\begin{align}\label{eq_distribution_Cm}
\Pb{(C_1=c_1,\cdots,C_n=c_n)}=\prod_{m=1}^n \left(\frac{1}{m}\right)^{c_m} \frac{1}{c_m!}.
\end{align}
\end{lemma}
%
%
%
%
%
%
\begin{lemma}
\label{lem_cm_rw_ym}
The random variables $C^{(n)}_m$ converge for each $m\in\Nr$ in distribution to a Poisson-distributed random variable $Y_m$ with $\E{Y_m}=\frac{1}{m}$.
In fact, we have for all $b\in\Nr$
$$(C_1^{(n)},C_1^{(n)},\cdots,C^{(n)}_b) \xrightarrow{d} (Y_1,Y_2,\cdots,Y_b) \qquad (n\rw\infty), $$
with all $Y_m$ independent.
\end{lemma}
%

Since $W^{1,n}(f)(x)$ and $W^{1,n+1}(f)(x)$ are defined on different spaces, it is difficult to compare them. Fortunately, the Feller coupling constructs a probability space and new random variables $C_m^{(n)}$ and $Y_m$ on this space, which have the same distributions as the $C_m^{(n)}$ and $Y_m$ above and can easily  be compared. Many more details on the Feller coupling can be found in \cite{barbour}.

The construction works as follows: Let $\xi:=(\xi_1\xi_2\xi_3\xi_4\xi_5\cdots)$ be a sequence of independent Bernoulli-random variables with $\E{\xi_m}=\frac{1}{m}$. An $m-$spacing is a sequence of $m-1$ consecutive zeroes in $\xi$ or its truncations: $$1\underbrace{0\cdots0}_{m-1\text{ times}}1.$$
\begin{definition}
\label{def_cm_feller}
Let $C_m^{(n)}(\xi)$ be the number of m-spacings in $1\xi_2\cdots \xi_n 1$.
We define $Y_m(\xi)$ to be the number of m-spacings in the whole sequence $\xi$.
\end{definition}
\begin{theorem}
\label{thm_feller_conv}We have
          \begin{itemize}
            \item The above-constructed $C_m^{(n)}(\xi)$ have the same distribution as the $C_m^{(n)}(\lambda)$ in definition \ref{def_cm_cycle}.
            \item $Y_m(\xi)$ is Poisson-distributed with $\E{Y_m(\xi)}=\frac{1}{m}$ and all $Y_m(\xi)$ are independent.
            \item $\E{\left|C_m^{(n)}(\xi)-Y_m(\xi)\right|}\leq \frac{2}{n+1}$.
            \item For any fixed $b\in\Nr$,
            $$\Pb{(C_1^{(n)}(\xi),\cdots, C_b^{(n)}(\xi))\neq(Y_1(\xi),\cdots ,Y_b(\xi))}\rw 0\ (n\rw\infty).$$
          \end{itemize}
\end{theorem}
\begin{proof}
For a full proof, we refer the reader to \cite{barbour}. We only need to mention for what follows the bijection between cycle-types for permutations in $\mathcal{S}_n$ and sequences of 0s and 1s $1\xi_2\cdots \xi_n 1$, obtained by writing a $1$ for closing a cycle and $0$ otherwise. For instance, the permutation $(1,2,3,4,5,6,7,8,9)(10,11,12)(13,14)(15,16)(17)$ in $\mathcal{S}_17$, of  cycle-type $(9,3,2,2,1)$ is mapped to the sequence $10000000100101011$. One remarks that $m$-spacings then correspond to $m$-cycles.
\end{proof}
%

%
We will use in the rest of this section only the random variables $C_m^{(n)}(\xi)$ and $Y_m(\xi)$. We thus just write $C_m^{(n)}$ and $Y_m$ for them. One might guess from the definition that $C_m^{(n)}\leq Y_m$, but this is not true. It is indeed possible that $C_m^{(n)}= Y_m+1$, but this can only happen if $\xi_{n-m}\cdots\xi_{n+1}=10\cdots 0$. If $n$ is fixed, we have at most one such $m$ with $C_m^{(n)}=Y_m+1$. In order to state the following lemma, we set
\begin{align}\label{eq_def_Bm}
B_m^{(n)}:=\set{\xi : \xi_{n-m}\cdots\xi_{n+1}=10\cdots 0}.
\end{align}
\begin{lemma}
\label{lem_properties_of_feller}
We have:
\begin{itemize}
        \item $C_m^{(n)}\leq Y_m+\mathbf{1}_{B_m^{(n)}}$
        \item $\Pb{B_m^{(n)}}=\frac{1}{n+1}$
	\item $C_m^{(n)}$ does not converge a.s.~against $Y_m$.
\end{itemize}
\end{lemma}
\begin{proof}
The first point follows form the above considerations.
The second point is a simple calculation using the independence of $\xi_i$.
We illustrate the proof of the last point with an example.
Let $\xi=(100010001\cdots)$ be given. Then
\begin{center}
\begin{tabular}{|c|c|c|c|c|c|c|c|c|}
  \hline
  $n$         & 1 & 2 & 3 & 4 & 5 & 6 & 7 & 8  \\ \hline
  $C_1^{(n)}$ & 1 &   &   &   & 1 &   &   &    \\
  $C_2^{(n)}$ &   & 1 &   &   &   & 1 &   &    \\
  $C_3^{(n)}$ &   &   & 1 &   &   &   & 1 &    \\
  $C_4^{(n)}$ &   &   &   & 1 & 1 & 1 & 1 & 2  \\
  \hline
\end{tabular}
\end{center}
The general case is completely similar to this example.
Let $\xi_v\xi_{v+1}\cdots \xi_{v+m+1}=10\cdots01$. We then have for $1\leq{m_0}\leq m-1$ and $v\leq n\leq v+m+1$
$$C_{m_0}^{(n)}=\left\{
              \begin{array}{ll}
                C_{m_0}^{(v)}+1, & \hbox{if } n=v+m_0,\\
                C_{m_0}^{(v)}, & \hbox{if } n\neq v+m_0.
              \end{array}
            \right.
$$
Since all $Y_m<\infty$ a.s. and $\sum_{m=1}^\infty Y_m=\infty$ a.s.~we are done.
\end{proof}

\subsection{The limit distribution}
\label{sec_the_limit}
In this subsection we write down a possible limit of $W^1(f)(x)$ and show that it is a good candidate.
%
%

If a $\sigma\in \mathcal{C}_\la$ with $|\la|=n$ is given then
\begin{align*}
W^{1,n}(f)(x)=\prod_{m=1}^{n}f(x^m)^{C_m^{(n)}}.
\end{align*}
We know that $C_m^{(n)}\xrightarrow{d} Y_m$ and so a natural and  possible limit would be
\begin{align}\label{eq_w1_infty_g_Ym}
f_\infty(x):=W^{1,\infty}(f)(x):=\prod_{m=1}^{\infty}f(x^m)^{Y_m}.
\end{align}
We prove in lemma~\ref{lem_Zn rw Z_infty} that $W^{1,n}(f)\xrightarrow{d} W^{1,\infty}(f)$.
Of course there are many things we need to check. We start with
\begin{lemma}
\label{lem_g_infty_holo}
The function $f_\infty(x)$ is a.s.~a holomorphic function of $x \in B_{r}(0)$.
\end{lemma}
\begin{proof}
We first mention that
$\log\left((1-x)^{m}\right)\equiv m \log\left(1-x\right)\mod 2\pi i$ for $m\in \Nr$.
The argument of $\log$ is always in $[-\pi,\pi]$ and therefore
$$\left|\log\left((1-x)^{m}\right)\right|\leq  m \left|\log\left(1-x\right)\right|\text{ for }m\in\Nr.$$
We write next $f(x)=1+xh(x)$ with $h$ holomorphic in $B_{r}(0)$.
Choose $m_0\in\Nr$ such that $|x^m h(x^m)|<r<1$ for all $m\geq m_0$ and all $x\in B_{r}(0)$.
We define $\delta=\delta(r)=\sup_{x\in B_{r}(0)} |h(x)|$ and remember the general fact that there exists $\beta=\beta(r)$ with $\left|\log(1+x)\right|\leq \beta |x|$ for $|x|<r$. We get
\begin{multline*}
\left|\log\left(\prod_{m=m_0}^\infty f(x^m)^{Y_m}\right)\right|
\leq \sum_{m=m_0}^\infty Y_m |\log\bigl(1+x^mh(x^m)\bigr)|\\
\leq \beta\sum_{m=m_0}^\infty Y_m |x^mh(x^m)|
\leq \delta \beta \sum_{m=m_0}^\infty Y_m r^m.
\end{multline*}
It remains to show that the last sum is a.s.~finite. This can be shown with the Borel-Cantelli theorem (see \cite{DZ}). Therefore $\prod_{m=m_0}^\infty f(x^m)^{Y_m}$ is a.s. a holomorphic function in $x$ and so is $f_\infty(x)$.
\end{proof}
We have proven that $f_\infty(x)$ is a.s.~a holomorphic function. This does not imply the holomorphicity of $\E{f_\infty(x)}$, even when it exists. We therefore prove
\begin{lemma}
\label{lem_E(g_infty)_holo_and_value}Let $f(x) := \sum_k b_kx^k$ and $x\in B_r(0)$. Then all moments of $f_\infty(x)$ exist.
The expectation $\E{f_\infty(x)}$ is a holomorphic function on $B_r(0)$ with
$$\E{f_\infty(x)}=\prod_{k=1}^\infty \frac{1}{(1-x^k)^{b_k}}.$$
\end{lemma}
\begin{proof}
\begin{description}
\item[Step 1] We show that $\E{f_\infty(x)}$ exists.
We define $h$ and $\delta$ as above and obtain
\begin{multline*}
\E{\left|\prod_{m=1}^\infty f(x^m)^{Y_m}\right|}
=
\E{\prod_{m=1}^\infty \left|\bigl(1+x^mh(x^m)\bigr)^{Y_m}\right|}
\leq
\E{\prod_{m=1}^\infty (1+\delta r^m)^{Y_m}}\\
\substack{(*)\\=}
\prod_{m=1}^\infty \exp\left(\frac{(1+\delta r^m)-1}{m}\right)
=\exp\left(\delta \sum_{m=1}^\infty \frac{r^m}{m}\right)<\infty,
\end{multline*}
where in $\substack{(*)\\=}$ we have used that
         $$\E{y^{(Y_m)}}=\exp\left(\frac{y-1}{m}\right)\text{ for }y\geq0$$
when $Y_m$ is a Poisson distributed random variable with $\E{Y_m}=\frac{1}{m}$. This can be shown by a simple calculation, expanding the exponential series.
 This proves the existence of $\E{f_\infty(x)}$.
\item[Step 2] We calculate the value of $\E{f_\infty(x)}$:
\begin{multline*}
\E{f_\infty(x)}
=
\E{\prod_{m=1}^\infty f(x^m)^{Y_m}}
=
\prod_{m=1}^\infty \E{f(x^m)^{Y_m}}
=
\prod_{m=1}^\infty \exp\left(\frac{f(x^m)-1}{m}\right)\\
=
\exp\left(\sum_{m=1}^\infty \frac{1}{m}\sum_{k=1}^\infty b_k x^{mk} \right)
=
\exp\left(\sum_{k=1}^\infty b_k \sum_{m=1}^\infty \frac{x^{km}}{m} \right)\\
=
\exp\left(\sum_{k=1}^\infty b_k \bigl(-\log(1-x^k)\bigr)\right)
=
\prod_{k=1}^\infty \frac{1}{(1-x^k)^{b_k}}.
\end{multline*}
The exchange of the exponential and the product in the first line is justified by step 1 and the exchange of the two sums as follows: the convergence radius of the Taylor expansion of $f$ is at least $r_1$ and so
$$\limsup_{k\rw\infty}|b_k|^{1/k}
\leq
\frac{1}{r_1}<\frac{1}{r}$$
Therefore there exists a constant $C$ with $|b_k|<C(\frac{1}{r_1})^k$. We define $r':=\frac{1}{r_1}$.
Clearly $r'r<1$, and so
\begin{eqnarray*}
\sum_{m=1}^\infty \sum_{k=1}^\infty \left|b_k \frac{x^{mk}}{m}\right|
\leq
C\sum_{m=1}^\infty \sum_{k=1}^\infty \frac{1}{m}(r')^kr^{mk}
=
C\sum_{m=1}^\infty \sum_{k=1}^\infty \frac{1}{m}(r'r^m)^k
=
C\sum_{m=1}^\infty \frac{1}{m} \frac{r'r^m}{1-r'r^m}
<\infty.
\end{eqnarray*}
\item[Step 3] Holomorphicity of $\E{f_\infty(x)}$:
\begin{align*}
\left|\log\left(\prod_{k=1}^\infty \frac{1}{(1-x^k)^{b_k}}\right)\right|
&\leq
\sum_{k=1}^\infty |b_k| |\log(1-x^k)|
\leq
\beta \sum_{k=1}^\infty |b_k| |x^k|
<
\infty
\end{align*}
\item[Step 4] Existence of the moments. Let $\sigma \in \mathcal{C}_\lambda$.
We have
\begin{multline*}
\Bigl(W^1(f_1)(x) W^1(f_2)(x)\Bigr)(\sigma)
=
\left(\prod_{m=1}^{l(\la)} f_1(\theta_m x^{\la_m})\right) \left(\prod_{m=1}^{l(\la)} f_2(\theta_m x^{\la_m})\right)
=\\
\prod_{m=1}^{l(\la)} \Bigl(f_1(\theta_m x^{\la_m}) f_2(\theta_m x^{\la_m})\Bigr)
=
\prod_{m=1}^{l(\la)} (f_1 f_2)(\theta_m x^{\la_m})
=
\left(W^1(f_1f_2)(x)\right) (\sigma),
\end{multline*}
and step 4 now follows from steps 1, 2, and 3.
\end{description}
\end{proof}


\subsection{Convergence against the limit}
\label{sec_conv_for_x<1_holo}
We give in this section two proofs of theorem~\ref{thm_conv_x<1_holo}. The idea of the first proof is to show $W^{1,n}(f)(x)\xrightarrow{d}W^{1,\infty}(f)(x)$ and then use uniform integrability. The idea of the second is to show that $\E{W^{1,n}(f)(x)}\rw \E{W^{1,\infty}(f)(x)}$ for $x\in[0,r']$ with $r'$ small enough and then to apply the theorem of Montel.

Note that the second proof does not imply $W^{1,n}(f)(x)\xrightarrow{d}W^{1,\infty}(f)(x)$. One would need that $W^{1,n}(f)(x)$ is uniquely determined by its moments. One possibility to check this is \emph{Carleman's condition} (see \cite{Billingsley}). Unfortunately it is very difficult to apply it directly to $W^{1,\infty}(f)(x)$, it is much easier to apply it to the upper bound $F(r)$ (see below). The problem is that $F(r)$ does not fulfill Carleman's condition.

\subsubsection{First proof of theorem~\ref{thm_conv_x<1_holo}}
We first prove
\begin{lemma}
\label{lem_Zn rw Z_infty}
Choose an $0<u\ (\leq r)$ such that $f(x)\neq 0$ for $x\in B_{u}(0)$.
We have for all fixed $x\in B_{u}(0)$
\begin{align}
\sum_{m=1}^n C_m^{(n)}\log\bigl(f(x^m)\bigr)
&\xrightarrow{d}
\sum_{m=1}^\infty Y_m \log\bigl(f(x^m)\bigr) \qquad (n\rw\infty)\\
W^{1,n}(f)(x)
&\xrightarrow{d}
W^{1,\infty}(f)(x) \qquad (n\rw\infty)
\end{align}
\end{lemma}
\textbf{Remark}: While the function $\sum_{m=1}^n C_m^{(n)}\log\bigl(f(x^m)\bigr)$ is not guaranteed to be holomorphic in $x$, it is well-defined with the convention $\log(-y):=\log(y)+i\pi$.

\begin{proof}
Since the exponential map is continuous, the second part follows immediately from the first part (a proof of this fact can be found in \cite{Billingsley}.)
We know from theorem~\ref{thm_feller_conv} that
\begin{align}
\E{\left|C_m^{(n)}-Y_m\right|}\leq \frac{2}{n+1}.
\end{align}
We use $\delta, \beta$ and $h$ as in the proof of lemma~\ref{lem_g_infty_holo}
to get
\begin{multline*}
\E{\left|\sum_{m=1}^n (Y_m-C_m^{(n)})\log\bigl(f(x^m)\bigr)\right|}
\leq
\sum_{m=1}^n\E{|Y_m-C_m^{(n)}|}|\log\bigl(f(x^m)\bigr)|\\
\leq
\sum_{m=1}^n \frac{2}{n+1}|\log\bigl(1+x^m h(x^m)\bigr)|
\leq
 \frac{2}{n+1}\left(C+\sum_{m=k_0}^n \beta (1+\delta) r^m\right)
\longrightarrow
0 \quad (n\rw\infty).
\end{multline*}
\end{proof}

Weak convergence does not automatically imply  convergence of the expectation. One need some additional properties. We introduce therefore
\begin{definition}
A sequence of (complex valued) random variables $(X_m)_{m\in\Nr}$ is called uniformly integrable if
$$\sup_{n\in\Nr} \E{|X_n|\mathbf{1}_{|X_n|>c}}\longrightarrow 0 \text{ for } c\rw \infty.$$
\end{definition}
One can now use
\begin{lemma}
[see \cite{gut}]
\label{lem_uniform integrabel}
Let $(X_m)_{m\in\Nr}$ be uniformly integrable and assume that $X_n\xrightarrow{d} X$. Then,
$$\E{X_n} \longrightarrow \E{X}.$$
\end{lemma}

We now finish the proof of theorem~\ref{thm_conv_x<1_holo}.
We define $h$ and $\delta$ as in the proof of lemma~\ref{lem_g_infty_holo} and get
$$|W^{1,\infty}(f)(x)|=\prod_{m=1}^\infty |f(x^m)|^{Y_m}\leq \prod_{m=1}^\infty (1+\delta r^m)^{Y_m}=:\widetilde{F}(r).$$
It is possible that $C_m^{(n)}=Y_m+1$ and so $\widetilde{F}(r)$ is not automatically an upper bound for $W^{1,n}(f)$, but $F(r):= \prod_{m=1}^\infty (1+\delta r^m)^{Y_m+1}$ is. We now have
$|W^{1,n}(f)|\leq F(r)$ and $\E{|W^{1,n}(f)|}\leq \E{F(r)}.$
We get with this $F(r)$
$$
\sup_{n\in\Nr} \E{ \left| W^{1,n}(f)(x)\right| \mathbf{1}_{\left| W^{1,n}(f)(x)\right|>c}}
\leq
\E{\left|F(r)\right|\mathbf{1}_{F(r)>c}} \longrightarrow 0 \qquad (c\rw\infty)
$$
The sequence $ \Bigl(W^{1,n}(f)(x)\Bigr)_{n\in\Nr}$ is therefore uniformly integrable.
Lemmas~\ref{lem_Zn rw Z_infty} and \ref{lem_uniform integrabel} together prove theorem~\ref{thm_conv_x<1_holo} for $x\in B_{u}(0)$ with $u$ as in lemma~\ref{lem_Zn rw Z_infty}. $\E{W^{1,n}(f)(x)}$ and $\E{W^{1,\infty}(f)(x)}$ are holomorphic in $B_r(0)$ and bounded by $\E{F(r)}$. Therefore theorem~\ref{thm_conv_x<1_holo} is true for all $x\in B_r(0)$.\qed
\subsubsection{Second proof of theorem~\ref{thm_conv_x<1_holo}}
We first prove a special case.
\begin{lemma}
\label{lem_conv_holo_reel}
Assume that $b_k\in\Rr$ for $1\leq k<\infty$ and choose $0<r'<r$ such that either $f(x)\leq 1$ on the interval  $x\in [0,r']$ or $f(x)\geq 1$. We then have for $x\in [0,r']$
$$\E{W^{1,n}(f)(x)}\rw \E{f_\infty(x)}.$$
\end{lemma}
\textbf{Remark}: We can choose such an $r'$ because $f$ is holomorphic.
\begin{proof}
\hfill
\begin{description}
\item[{Case $f(x)\leq 1$ for $x \in [0,r'] $} ] \hfill
\begin{description}
\item[Inequality $\leq$]
Choose $m_0\in\Nr$ arbitrary. We have for $n\geq m_0$
$$\E{\prod_{m=1}^n f(x^m)^{C_m^{(n)}}}
\leq \E{\prod_{m=1}^{m_0} f(x^m)^{C_m^{(n)}}}
$$
We know that $(C_1^{(n)},\cdots, C_{m_0}^{(n)}) \xrightarrow{d} (Y_1,\cdots, Y_{m_0})$ and $\prod_{m=1}^{m_0} f(x^m)^{C_m}\leq 1$.
This is enough to give
$$\E{\prod_{m=1}^{m_0} f(x^m)^{C_m^{(n)}}} \rw \E{\prod_{m=1}^{m_0} f(x^m)^{Y_m}} \text{ for }n\rw\infty.$$
Therefore,
$$\limsup_{n\rw \infty} \E{\prod_{m=1}^n f(x^m)^{C_m^{(n)}}}
\leq \inf_{m_0\in\Nr} \E{\prod_{m=1}^{m_0} f(x^m)^{Y_m}}
=\E{f_\infty(x)}$$
\item[Inequality $\geq$]
We need in this case the Feller coupling. If we would have always $C_m^{(n)}\leq Y_m$ then we would have no problem.
We have in fact $C_m^{(n)}\leq Y_m$ if $\xi_{n+1}=1$. We use therefore a small trick.
Remember the definition
$B_k=B_k^{(n)}=\set{\xi_{n-k}\cdots\xi_{n+1}=100\cdots 0}$.
We write
\begin{align*}
W^{1,n}(f)(x)
&=
\prod_{m=1}^n f(x^m)^{C_m^{(n)}}
=
\left(\sum_{k=0}^n  \textbf{1}_{B_k}\right)\prod_{m=1}^n f(x^m)^{C_m^{(n)}}\\
&=
\prod_{m=1}^n f(x^m)^{C_m^{(n)}}\textbf{1}_{B_0}+\sum_{k=1}^n f(x^k)\prod_{m=1}^n f(x^m)^{C_m^{(n-k)}}\textbf{1}_{B_k}\\
&=
W^{1,n}(f)(x)\textbf{1}_{B_0}+
\sum_{k=1}^n f(x^k) W^{1,n-k}(f)(x)\textbf{1}_{B_k}.
\end{align*}
Let $0<\eps<1$ be arbitrary and fixed. Since $f(0)=1$ there exists a $k_0$ such that $1-\eps<f(x^k)\leq1$ for $k\geq k_0$. We get
\begin{align*}
\E{W^{1,n}(f)}
&=\E{W^{1,n}(f)(x)\textbf{1}_{B_0}} +
\E{\sum_{k=1}^n f(x^k) W^{1,n-k}(f)(x)\textbf{1}_{B_k}}\\
&\geq
\E{W^{1,\infty}(f)(x)\textbf{1}_{B_0}} +
\E{\sum_{k=1}^n f(x^k) W^{1,\infty}(f)(x)\textbf{1}_{B_k}}\\
&\geq
\E{W^{1,\infty}(f)(x)\textbf{1}_{B_0}} +
\E{\sum_{k=1}^{k_0} f(x^k) W^{1,\infty}(f)(x)\textbf{1}_{B_k}}\\
&+\E{\sum_{k=k_0+1}^n (1-\eps) W^{1,\infty}(f)(x)\textbf{1}_{B_k}}.
\end{align*}
We have that $\Pb{B_k}=\frac{1}{n+1}$, $f(x^k)$ is bounded for $1\leq k\leq k_0$ and all moments of $W^{1,\infty}(f)$ exist.
We can now apply the Schwarz inequality (for $\mathrm{L^2}$) to see that the first two summands go to $0$. We can replace $k_0$ by $0$ in the third summand by the same argument.
\end{description}
\item[{Case $f(x)\geq 1$ for $x\in [0,r']$}]
The arguments are almost the same. We have to exchange all $\leq$ and $\geq$ and check that
\begin{align}\label{eq_compare_g_m0}
\E{\prod_{m=1}^{m_0} f(x^m)^{C_m^{(n)}}}
\longrightarrow
\E{\prod_{m=1}^{m_0} f(x^m)^{Y_m}}
\end{align}
Since $C_m^{(n)}\leq Y_m+1$, we have a common (integrable) upper bound. We also know from theorem~\ref{thm_feller_conv} that
$$\Pb{(C_1^{(n)},\cdots, C_{m_0}^{(n)})\neq(Y_1,\cdots ,Y_{m_0})}\rw 0\ (n\rw\infty)$$
These two facts together prove \eqref{eq_compare_g_m0}.
\end{description}
\end{proof}

We now extend lemma~\ref{lem_conv_holo_reel} to arbitrary $x$ and $b_k$.
We have constructed in the proof of lemma~\ref{lem_conv_holo_reel} an lower and an upper bound for $\E{W^{1,n}(f)(x)}$ and showed that they converge to the same limit as $n\rw\infty$. One could try to modify this proof to apply it to general $x$ and $b_k$, but this is rather difficult. It is easier to use the theorem of Montel (see \cite{fritzsche}).
\begin{lemma}
\label{lem_conv_holo_complex}
We have for any holomorphic function $f$ and $x\in B_r(0)$
$$\E{W^{1,n}(f)(x)}\rw \E{f_\infty(x)}\text{ for }n\rw\infty.$$
\end{lemma}
\begin{proof}\hfill
\begin{description}
\item[Step 1]
We show first that lemma~\ref{lem_conv_holo_reel} is true for arbitrary $x\in B_r(0)$ and all $b_k\in\Rr$ (i.e.~with no condition that either $f(x)\ge1$ or $f(x) \leq 1$ on a whole interval).
We use the $F(r)$ from the first proof. We apply the theorem of Montel with the upper bound $\E{F(r)}$ for $\E{W^{1,n}(f)(x)}$.\\
Suppose that there exists a $x_0\in B_r(0)$ where $\E{W^{1,n}(f)(x_0)}\nrightarrow \E{W^{1,\infty}(f)(x_0)}$. Then there exists a $\eps>0$ and a sequence $\Lambda\subset\Nr$ with $|\E{W^{1,n}(f)(x_0)}-\E{W^{1,\infty}(f)(x_0)}|>\eps$ for $n\in \Lambda$.
We apply the theorem of Montel and get a subsequence $\Lambda'\subset \Lambda$ and a holomorphic function $h$ with
$$\E{W^{1,n}(f)(x)}\longrightarrow h(x) \text{ for }x\in B_r(0)\text{ and }n\in\Lambda'.$$
We know form lemma~\ref{lem_conv_holo_reel} that $h$ has to agree with $\E{f_\infty(x)}$ on $[0,r']$. This is a contradiction, since $\E{f_\infty(x_0)}\neq h(x_0)$.
\item[Step 2]
We now prove the lemma. Let $b_k$ be arbitrary. We define $b_k(s):=\Re(b_k)+s\Im(b_k)$ and
$$f^{(s)}(x):=1+\sum_{k=1}^\infty b_k(s) x^k.$$
It follows as in lemmas~\ref{lem_g_infty_holo} and \ref{lem_E(g_infty)_holo_and_value} that $W^{1,\infty}\left(f^{(s)}\right)(x)$ and $\E{W^{1,\infty}\left(f^{(s)}\right)(x)}$ are holomorphic functions in $x$ and $s$.
We know from step 1 that
$$\E{W^{1,n}\left(f^{(s)}\right)(x)}\longrightarrow \E{W^{1,\infty}\left(f^{(s)}\right)(x)}\qquad \text{ for } x\in B_r(0), s\in\Rr.$$
We use as in step 1 the theorem of Montel and get
$$\E{W^{1,n}\left(f^{(s)}\right)(x)}\longrightarrow \E{W^{1,\infty}\left(f^{(s)}\right)(x)}\qquad \text{ for } x\in B_r(0), s\in\C.$$
We now put $s=i$ and are done.
\end{description}
\end{proof}
Putting lemmas~\ref{lem_E(g_infty)_holo_and_value} and \ref{lem_conv_holo_complex} together, we have proved that for any holomorphic function $f$ and $x\in B_r(0)$
$$\E{W^{1,n}(f)(x)}\rw \E{f_\infty(x)} = \prod_{k=1}^\infty \frac{1}{(1-x^k)^{b_k}},$$
i.e.~ the special case of theorem~\ref{thm_conv_x<1_holo} in just one variable and when $b_{0,0}=1$. By lemma~\ref{lem_reduction_to_1}, this is enough to prove the full theorem (in one variable).%
\subsection{Proof of theorem~\ref{thm_conv_x<1_holo} for $W^2$}
\label{sec_Remarks W^2}
The proof of theorem~\ref{thm_conv_x<1_holo} for $W^2$ is almost the same.
One only has to replace $W^{1,\infty}(f)(x)$ with
\begin{align}\label{eq_w2_infty_g_Ym}
W^{2,\infty}(f)(x):=\prod_{m=1}^{\infty}f(x^m,\alpha_1^m,\alpha_2^m,\cdots)^{Y_m}.
\end{align}

\section{Other Groups}
\label{sec_other_groups}
We can also use the techniques of this paper  for some other groups than $\mathcal{S}_n$.
These are the alternating group and the Weyl groups of classical groups.
We do not give here the definition of a Weyl group, since this would go too far and can be found in many books about Lie groups, for instance in \cite{bump}.
We will only give a presentation of the group and write down the generating functions.
The asymptotic behaviour follows directly form theorem~\ref{thm_conv_x<1_poly} and theorem~\ref{thm_conv_x<1_holo}.

\subsection{The alternating group $\mathcal{A}_n$}
\label{sec_alternating group}
It is natural to ask if we can use the techniques of section \ref{sec_wreath_prod_poly} to obtain generating functions for  subgroups of $\mathcal{S}_n$, and section \ref{sec_wreath_prod_poly} is based on \eqref{E_f_class_n}. Since this formula is only true for class functions on
$\mathcal{S}_n$, the possible subgroups have to be normal. Therefore the only candidate is the alternating group $\mathcal{A}_n$.

\subsubsection{Definitions}
\begin{definition}
A $\sigma\in \mathcal{S}_n$ is called \emph{even} if $\sigma$ can be written as an even number of transpositions. Otherwise
$\sigma$  is called \emph{odd}. The \emph{alternating group} $\mathcal{A}_n$ is the subset of $\mathcal{S}_n$ of all even permutations. The signature $\eps(\sigma)$ of a permutation is 1 for even permutations, -1 for odd ones.
\end{definition}
\begin{lemma}
 The signature $\eps$ is a group homomorphism and $\mathcal{A}_n = \ker(\eps)$.
\end{lemma}
\begin{definition}
\label{def_haar_An}
We write $\EAn{f}$ for the expectation with respect to the Haar-measure $\mu_{\mathcal{A}_n}$ on $\mathcal{A}_n$.
Explicitly we have for $n\geq2$ (only, because $\mathcal{S}_1=\mathcal{A}_1=\set{1}$)
\begin{align}\label{eq_def_E_An}
\EAn{f}=\frac{2}{n!} \sum_{\sigma\in \mathcal{A}_n} f(\sigma).
\end{align}
\end{definition}
\subsubsection{Generating functions for $W^1$ and $W^2$ on $\mathcal{A}_n$}

We prove in this subsection
\begin{theorem}
\label{thm_gen_fn_poly_An}
Let $\theta$ and $\vth$ be random variables with values in $S^1$ and $P$ be a polynomial with
$$P(x_1,x_2)=\sum_{k_1,k_2=0}^{\infty}b_{k_1,k_2} x_1^{k_1} x_2^{k_2}. $$
We define as in \eqref{eq_def_alpha}
$\alpha_{k_1,k_2}:=\E{\theta^{k_1}\vth^{k_2}}$.
We have for $n\geq 2$
\begin{multline}\label{eq_thm_gen_W1_poly_An}
\EAn{W^1(P)(x_1,x_2)}
\\=
\nth{\prod_{k_1,k_2=0}^\infty (1-x_1^{k_1} x_2^{k_2}t)^{-b_{k_1,k_2} \alpha_{k_1,k_2}} }
+
\nth{\prod_{k_1,k_2=0}^\infty (1+x_1^{k_1} x_2^{k_2}t)^{-b_{k_1,k_2} \alpha_{k_1,k_2}} }
\end{multline}
\begin{multline}\label{eq_thm_gen_W2_poly_An}
\EAn{W^2(P)(x_1,x_2)}
\\=
\nth{\prod_{k_1,k_2=0}^\infty (1-\alpha_{k_1,k_2} x_1^{k_1} x_2^{k_2}t)^{-b_{k_1,k_2}} }
+
\nth{\prod_{k_1,k_2=0}^\infty (1+\alpha_{k_1,k_2} x_1^{k_1} x_2^{k_2}t)^{-b_{k_1,k_2}} }
\end{multline}
The products in \eqref{eq_thm_gen_W1_poly_An}, \eqref{eq_thm_gen_W2_poly_An} are holomorphic for $|t|<1$ and $\max\set{|x_1|,|x_2|}\leq 1$.
\end{theorem}
The idea of the proof is to reformulate $\EAn{\cdot}$ and to use the results of section \ref{sec_wreath_prod_poly}.
We start with
\begin{lemma}
\label{lem_An=Sn+epsSn}
We have for each $f:\mathcal{S}_n\rw \C$
\begin{align}\label{eq_An=Sn+epsSn}
\EAn{f|_{\mathcal{A}_n}}=\En{f}+\En{\eps\cdot f} \text{ for }n\geq2
\end{align}
and $\eps(\sigma) = \prod_{m=1}^{l(\la)}(-1)^{\la_m+1}$ for $\sigma\in \mathcal{C}_\la$.
\end{lemma}
\begin{proof}
We have for $n\geq2$
$$
\EAn{f|_{\mathcal{A}_n}}
=
\frac{2}{n!}\sum_{\substack{\sigma\in \mathcal{S}_n\\\eps(\sigma)=1}} f(\sigma)
=
\frac{1}{n!}\sum_{\sigma\in \mathcal{S}_n} \bigl((1+\eps)f\bigr)(\sigma)
=
\En{f}+\En{\eps\cdot f}
$$
This proves \eqref{eq_An=Sn+epsSn}. The second statement is trivial.
\end{proof}
We can now prove theorem \ref{thm_gen_fn_poly_An}:
\begin{proof}[Proof of theorem \ref{thm_gen_fn_poly_An}]
We calculate $\sum_{n=0}^\infty \En{\eps W^1(P)}t^n$ and use lemma \ref{lem_An=Sn+epsSn}.
This calculation is very similar to the calculations in the proof of theorem \ref{thm_gen_fn_poly}.
We therefore simplify the proof by assuming $\theta\equiv\vth\equiv1$ and that $P$ is only dependent on one variable.
We get
\begin{multline*}
\sum_{n=0}^\infty \En{\eps W^1(P)(x)}t^n
=
\sla \prod_{m=1}^{l(\la)}(-1)^{\la_m+1}P(x^{\la_m})t^{|\la|}
\\=
\exp\left(\sum_{m=1}^\infty \frac{1}{m}(-1)^{m+1}P(x^m)t^m\right)
=
\exp\left(-\sum_{m=1}^\infty\frac{1}{m}(-1)^{m}\sum_{k=0}^\infty b_k x^{km}t^m\right)
=\\
\exp\left(-\sum_{k=0}^\infty b_k \log(1+x^kt)\right)
=
\prod_{k=0}^\infty (1+x^kt)^{-b_k}
\end{multline*}
\end{proof}

\subsection{The Weyl group of $SO(2n)$}
\label{sec_so(2n)}
Let $D$  be the set of diagonal matrices with diagonal entries $1,-1$.
The Weyl group $\mathcal{W}_n$ of $SO(2n)$ is equal to $D\mathcal{S}_n$. We define $Z_{\mathcal{W}_n}(x)(w):=\det(I-xw)$ for $w \in\mathcal{W}_n $.
We have to take a closer look at $\mu_{\mathcal{W}_n}$ to write down a generating function for the moments of $Z_{\mathcal{W}_n}(x)$.
Any element $w$ of $\mathcal{W}_n$ can be written uniquely as $w=dg$ with $d\in D,g\in \mathcal{S}_n$.
Therefore $\mu_{\mathcal{W}_n}=\mu_D\times\mu_{\mathcal{S}_n}$ and the diagonal matrices are independent of $\mathcal{S}_n$.
A simple calculation shows that the diagonal entries $d_i$ in $D$ are iid with $\Pb{d_i=1}=\Pb{d_i=-1}=\frac{1}{2}$.
This is precisely the definition of $W^2$ in \eqref{eq_def_W2_Zn_det}. We use theorem~\ref{thm_gen_fn_poly} and the example in section~\ref{ex_discrete} to get
\begin{align*}
\E{Z_{\mathcal{W}_n}^{s_1}(x_1)Z_{\mathcal{W}_n}^{s_2}(x_2)}
=\nth{
\prod_{\substack{k_1,k_2=0\\2|(k_1-k_2)}}^\infty (1-x_1^{k_1} x_2^{k_2} t)^{\binom{s_1}{k_1}\binom{s_2}{k_2}(-1)^{k_1+k_2+1}}}
\end{align*}

\subsection{The Weyl group of $SO(2n+1)$}
\label{sec_so(2n+1)}
Let $D$ be as above. The Weyl group $\mathcal{W}'_n$ of $SO(2n+1)$ is equal to $D\mathcal{A}_n$. We can argue as above and get with theorem~\ref{thm_gen_fn_poly_An}

\begin{multline*}
\E{Z_{\mathcal{W}'_n}^{s_1}(x_1)Z_{\mathcal{W}'_n}^{s_2}(x_2)}
=
\nth{\prod_{\substack{k_1,k_2=0\\2|(k_1-k_2)}}^\infty (1-x_1^{k_1} x_2^{k_2} t)^{\binom{s_1}{k_1}\binom{s_2}{k_2}(-1)^{k_1+k_2+1}}}
+\\
\nth{\prod_{\substack{k_1,k_2=0\\2|(k_1-k_2)}}^\infty (1+x_1^{k_1} x_2^{k_2} t)^{\binom{s_1}{k_1}\binom{s_2}{k_2}(-1)^{k_1+k_2+1}}}
\end{multline*}

\subsection{The Weyl group of $SU(n)$}
The Weyl group $\widetilde{\mathcal{W}}_n$ of $SU(n)$ is equal to $\mathcal{S}_n$, shrunk to the subspace $\mathcal{T}=\set{(x_1,\cdots,x_n)\in\C^n;\sum x_i=0}$.
It is easy to see that $\C^n \cong \mathcal{T}\oplus \C(1,1,\cdots,1)$ and the action of $\mathcal{S}_n$ on $\C(1,1,\cdots,1)$ is trivial.
Thus
\begin{align*}
\E{Z_{\widetilde{\mathcal{W}}_n}^{s_1}(x_1)Z_{\widetilde{\mathcal{W}}_n}^{s_2}(x_2)}
=
\nth{\frac{\prod_{k_1=0}^{s_1}\prod_{k_2=0}^{s_2}\left(1-x_1^{k_1}x_2^{k_2}t\right)^{\binom{s_1}{k_1}\binom{s_2}{k_2}(-1)^{k_1+k_2+1}}}
{(1-x_1)^{s_1}(1-x_2)^{s_2}}}
\end{align*}

\section{Further Questions}
%
We have proven several results in this paper, but there are a few questions left to consider.
\begin{itemize}
\item 
We have focused in this paper on the case $|x_i|<1$, but the generating functions in theorem~\ref{thm_gen_fn_poly} and theorem~\ref{thm_gen_holo} are also valid for $|x_i|=1$. This situation was already studied by the second author in \cite{DZ}, but only for for $f(x)=(1-x)^s$ with $s\in\Nr$ and $|x|=1$, $x$ not a root of unity. A partial fraction decomposition sufficed there to calculate the behaviour of $\En{Z_n^s(x)}$ for $n\rw\infty$. We cannot argue now in the same way since the generating functions are not rational functions anymore. One can instead use theorem VI.5 in \cite{FlSe09}. The problem is that this only works for generating functions associated to polynomials. It thus remains to determine the behaviour of $\En{W^j(f)(x)}$ for $n\rw\infty$ and $|x|=1$, $x$ not a root of unity.
\item Let $f(x) := 1/(1-x)$. It follows from theorem~\ref{thm_conv_x<1_holo} that
    $$
    \lim_{n\rw\infty} \En{W^j(f)(x)}
    =
    \prod_{k=1}^\infty \frac{1}{1-x^k}
    $$
    We put $x=e^{2\pi i \tau}$ with $\tau\in\mathcal{H}=\set{z\in\C; \Im(z)>0}$ and see that the product on the RHS is up to a factor $e^{2\pi i/24}$ the Dedekind eta function. It is now thus natural to look for functional equations satisfied by asymptotic expressions for coefficients of the generating functions associated to other $f$s.
\item The main result in \cite{HKOS} is that the real and the imaginary part of $\frac{\log\bigl(Z_n(x)\bigr)}{\sqrt{\frac{\pi}{12}\log(n)}}$ converges in distribution to normal distributed random variable (for $|x|=1$, $x$ not a root of unity). Does there exist a similar limit theorem for $W^{1,n}(f)$?
\end{itemize}

\thanks{
The authors wish to acknowledge Joseph Najnudel and Ashkan Nikeghbali for encouragements,  stimulating discussions and freely sharing their work in \cite{NN}.}

\bibliographystyle{plain}
\bibliography{wreath_product}

\end{document}